\documentclass[12pt]{article}

\usepackage{amsmath, amsthm, amssymb}
\usepackage{enumitem}
\usepackage{pdflscape}
\usepackage{caption}
\usepackage{bm}
\usepackage{ifpdf}
	\ifpdf
\usepackage[pdftex]{graphicx}
	\else
\usepackage[dvips]{graphicx}
	\fi
\usepackage[all]{xy}
\usepackage{tocvsec2}
\usepackage{bbm}
	\input xy
	\xyoption{all}
\usepackage[pdftex,plainpages=false,hypertexnames=false,pdfpagelabels]{hyperref}
	\newcommand{\arxiv}[1]{\href{http://arxiv.org/abs/#1}{\tt arXiv:\nolinkurl{#1}}}
	\newcommand{\arXiv}[1]{\href{http://arxiv.org/abs/#1}{\tt arXiv:\nolinkurl{#1}}}
	\newcommand{\doi}[1]{\href{http://dx.doi.org/#1}{{\tt DOI:#1}}}

\usepackage{tensor}
	
	\newcommand{\googlebooks}[1]{(preview at \href{http://books.google.com/books?id=#1}{google books})}
	
\usepackage{xcolor}
	\definecolor{dark-red}{rgb}{0.7,0.25,0.25}
	\definecolor{dark-blue}{rgb}{0.15,0.15,0.55}
	\definecolor{medium-blue}{rgb}{0,0,.8}
	\definecolor{DarkGreen}{RGB}{0,150,0}
	\definecolor{rho}{named}{red}
	\hypersetup{
	   colorlinks, linkcolor={purple},
	   citecolor={medium-blue}, urlcolor={medium-blue}
	}
\usepackage{longtable}
\usepackage{fullpage}

\usepackage{mathbbol}
\usepackage{lipsum}
\theoremstyle{plain}
\newtheorem{thm}{Theorem}[section]
\newtheorem*{thm*}{Theorem}
\newtheorem{thmalpha}{Theorem}

\newtheorem{cor}[thm]{Corollary}

\newtheorem*{cor*}{Corollary}

\newtheorem*{conj*}{Conjecture}
\newtheorem{lem}[thm]{Lemma}
\newtheorem{prop}[thm]{Proposition}

\newtheorem*{quest*}{Question}
\newtheorem*{claim*}{Claim}

\theoremstyle{definition}
\newtheorem{defn}[thm]{Definition}

\newtheorem{ex}[thm]{Example}
\newtheorem*{ex*}{Example}
\newtheorem{sub-ex}[thm]{Sub-Example}
\newtheorem{counter-ex}[thm]{Counter-Example}
\newtheorem{rem}[thm]{Remark}
\newtheorem*{rem*}{Remark}
\newtheorem{remark}[thm]{Remark}

\DeclareMathOperator{\Ad}{Ad}

\DeclareMathOperator{\End}{End}

\DeclareMathOperator{\Hom}{Hom}

\DeclareMathOperator{\op}{op}

\DeclareMathOperator{\spann}{span}

\DeclareMathOperator{\id}{id}

\DeclareMathOperator{\Irr}{Irr}

\DeclareMathOperator{\tr}{tr}

\DeclareMathOperator{\II}{II}

\newcommand{\comment}[1]{}

\newcommand{\bbOne}{\mathbbm{1}}
\newcommand\mapsfrom{\mathrel{\reflectbox{\ensuremath{\mapsto}}}}

\newcommand{\noshow}[1]{}
\newcommand{\MR}[1]{}

\newcommand{\Rep}{{\sf Rep}}

\newcommand{\Mod}{{\sf Mod}}

\newcommand{\Bim}{{\sf Bim}}

\newcommand{\fgpBim}{{\sf Bim_{fgp}}}

\newcommand{\fgpBimtr}{{\sf Bim_{fgp}^{tr}}}
\newcommand{\spbfBim}{{\sf Bim_{bf}^{sp}}}
\renewcommand{\Vec}{{\sf Vec}}

\newcommand{\Hilb}{{\sf Hilb}}

\newcommand{\rCorr}{{\mathsf{C^{*}Alg}}}

\newcommand{\vNA}{{\sf vNA}}

\newcommand{\CDisc}{{\sf C^*Disc}}

\newcommand{\<}{\langle}
\renewcommand{\>}{\rangle}
\newcommand{\cC}{\mathcal C}

\newcommand{\cB}{\mathcal B}
\newcommand{\bbD}{\mathcal D}
\newcommand{\cE}{\mathcal E}
\newcommand{\cF}{\mathcal F}
\newcommand{\cL}{\mathcal L}

\newcommand{\cX}{\mathcal X}
\newcommand{\cZ}{\mathcal Z}

\newcommand{\Clin}{{\mathsf{C^{*}lin}}}
\newcommand{\Surf}{{\mathsf{Surf}}}

\def\semicolon{;}
\def\applytolist#1{
    \expandafter\def\csname multi#1\endcsname##1{
        \def\multiack{##1}\ifx\multiack\semicolon
            \def\next{\relax}
        \else
            \csname #1\endcsname{##1}
            \def\next{\csname multi#1\endcsname}
        \fi
        \next}
    \csname multi#1\endcsname}

\def\calc#1{\expandafter\def\csname c#1\endcsname{{\mathcal #1}}}
\applytolist{calc}QWERTYUIOPLKJHGFDSAZXCVBNM;
\def\bbc#1{\expandafter\def\csname bb#1\endcsname{{\mathbb #1}}}
\applytolist{bbc}QWERTYUIOPLKJHGFDSAZXCVBNM;
\def\bfc#1{\expandafter\def\csname bf#1\endcsname{{\mathbf #1}}}
\applytolist{bfc}QWERTYUIOPLKJHGFDSAZXCVBNM;
\def\sfc#1{\expandafter\def\csname s#1\endcsname{{\sf #1}}}
\applytolist{sfc}QWERTYUIOPLKJHGFDSAZXCVBNM;
\def\fc#1{\expandafter\def\csname f#1\endcsname{{\mathfrak #1}}}
\applytolist{fc}QWERTYUIOPLKJHGFDSAZXCVBNM;

\def\-{{\hbox{-}}}
\usepackage{tikz}
\usepackage{tikz-cd}
\usetikzlibrary{arrows,backgrounds,patterns.meta}
\usetikzlibrary{positioning,shadings,cd}
\usetikzlibrary{shapes}
\usetikzlibrary{backgrounds}
\usetikzlibrary{decorations,decorations.pathreplacing,decorations.markings}
\usetikzlibrary{fit,calc,through}
\usetikzlibrary{external}
\usetikzlibrary{arrows}
\tikzset{vertex/.style = {shape=circle,draw,fill=black,inner sep=0pt,minimum size=5pt}}
\tikzset{edge/.style = {->,> = latex', bend right}}
\tikzset{
	super thick/.style={line width=3pt}
}
\tikzset{
    quadruple/.style args={[#1] in [#2] in [#3] in [#4]}{
        #1,preaction={preaction={preaction={draw,#4},draw,#3}, draw,#2}
    }
}
\tikzstyle{shaded}=[fill=red!10!blue!20!gray!30!white]
\tikzstyle{unshaded}=[fill=white]
\tikzstyle{empty box}=[circle, draw, thick, fill=white, opaque, inner sep=2mm]
\tikzstyle{annular}=[scale=.7, inner sep=1mm, baseline]
\tikzstyle{rectangular}=[scale=.75, inner sep=1mm, baseline=-.1cm]
\tikzstyle{mid>}=[decoration={markings, mark=at position 0.5 with {\arrow{>}}}, postaction={decorate}]
\tikzstyle{mid<}=[decoration={markings, mark=at position 0.5 with {\arrow{<}}}, postaction={decorate}]
\tikzstyle{over}=[double, draw=white, super thick, double=]

\tikzdeclarepattern{
  name=primeddots,
  type=uncolored,
  bounding box={(-.6pt,-.6pt) and (.6pt,.6pt)},
  tile size={(5pt,5pt)},
  tile transformation={rotate=60},
  parameters={none}, 
  code={
    \fill(0pt,0pt) circle (.35pt);
  }
}

\tikzdeclarepattern{
  name=primeddots2,
  type=uncolored,
  bounding box={(-.3pt,-.3pt) and (.3pt,.3pt)},
  tile size={(2.5pt,2.5pt)},
  tile transformation={rotate=60},
  parameters={none}, 
  code={
    \fill(0pt,0pt) circle (.35pt);
  }
}

\tikzstyle{primedregion}[none]=[
	preaction={fill=#1},
	pattern=primeddots,
  draw=#1,
]

\tikzstyle{primedregion2}[none]=[
	preaction={fill=#1},
	pattern=primeddots2,
  draw=#1,
]

\tikzcdset{scale cd/.style={every label/.append style={scale=#1},
    cells={nodes={scale=#1}}}}

\let\OLDthebibliography\thebibliography
\renewcommand\thebibliography[1]{
  \OLDthebibliography{#1}
  \setlength{\parskip}{0pt}
  \setlength{\itemsep}{0pt plus 0.3ex}
}

\begin{document}
\title{Inclusions of Operator Algebras from Tensor Categories: beyond irreducibility}
\author{Lucas Hataishi and Roberto Hern\'{a}ndez Palomares}

\newcommand{\Addresses}{{
  \bigskip
  \footnotesize

  Lucas~Hataishi, \textsc{University of Oxford}\par\nopagebreak
  \textit{E-mail address}: \texttt{lucas.hataishi@gmail.com}

  \medskip

  Roberto~Hern\'{a}ndez~Palomares , \textsc{University of Waterloo}\par\nopagebreak
  \textit{E-mail address}:
  \texttt{robertohp.math@gmail.com}

}}

\date{\today}

\maketitle

\begin{abstract}
     We derive faithful inclusions of C*-algebras from a coend-type construction in unitary tensor categories. This gives rise to different potential notions of discreteness for an inclusion in the non-irreducible case, and provides a unified framework that encloses the theory of compact quantum group actions. We also provide examples coming from semi-circular systems and from factorization homology. 
     In the irreducible case, we establish conditions under which the C*-discrete and W*-discrete conditions are equivalent.
\end{abstract}

\section{Introduction}
Inclusions of operator algebras provide a flexible framework to study classical and quantum symmetries of noncommutative spaces.  Jones' landmark discovery of the Index Rigidity Theorem for subfactors \cite{MR696688} launched the construction and classification of numerous new examples of subfactors, and fueled interactions with other fields such as low-dimensional topology, statistical mechanics, and quantum field theory. (See \cite{MR1473221} for a survey). 

Subsequently, \emph{discrete subfactors} were introduced in \cite{MR1622812} as a broad class of inclusions $N \subset M$ sharing properties of crossed products by outer actions of discrete groups on factors. They generally do not have finite index, but can be characterized either by a generating property in terms of the {\em quasi-normalizer} of $N$ in $M$, or by a decomposition of the space of bounded $N \- N$-bimodular operators on $L^2(M)$ as a product of type I factors.

Associated to a subfactor $N\subset M$ is its \emph{standard invariant}, $\cC_{N\subset M}$, that can be axiomatized using Jones' planar algebras \cite{MR4374438} or Popa's $\lambda$-lattices \cite{MR1334479}. The standard invariant classifies amenable subfactors of the hyperfinite $\rm II_1$-factor $\mathcal R$ \cite{MR1278111}.  The formulation of the standard invariant that is relevant to us follows the philosophy advocated in \cite{JP17,MR1966524,MR1966525}, and goes roughly as follows.

Let $N$ be a a $\II_1$-factor $N$, and denote by $\Bim(N)$ the category of $N\-N$-Hilbert bimodules. Suppose we are given a subfactor $N \subset M$. Consider the Hilbert space completion $L^2(M)$ with respect to the trace. Looking at it as a $N \-N$-bimodule, it decomposes into a direct sum of irreducible ones. Let $\cC_{N \subset M}$ be the subcategory of $\Bim(N)$ generated by those irreducible bimodules, their direct sums and  their fusion products. By construction, $\cC_{N \subset M}$ is a C$^*$-tensor subcategory of $\Bim(N)$. The factor $M$ induces a functor $\bbM: (\cC_{N \subset M})^{\op} \to \Vec$, defined on objects by $\bbM(K) = \Hom_{N-M}(K \underset{N}{\boxtimes} L^2(M),L^2(M))$. Then the standard invariant of $N \subset M$ is the pair $(\cC_{N \subset M}, \bbM)$. The main result in \cite{MR3948170} is that the standard invariants $(\cC_{N \subset M}, \bbM)$ classify extremal irreducible discrete extensions $N \subset M$. We say that $\ \cC_{N \subset M}$ acts on $N$ by means of the inclusion functor $ \cC_{N \subset M} \hookrightarrow \Bim(N)$. The procedure by which $M$ is reconstructed from the standard invariant is called the realization or {\em crossed product} of $N$ by $\bbM$, and it is denoted by $N \rtimes \bbM$.  We say that $\bbM$ is the {\em C$^*$-algebra object}, in the algebraic ind-completion $\Vec(\cC_{N \subset M}) := \{ \text{linear functors} \ F: \cC_{N \subset M}^{\op} \to \Vec \}$ of  $ \cC_{N \subset M}$, associated to the subfactor $N \subset M$.

In \cite{2023arXiv230505072H}, the second-named author and Nelson introduced the abstract properties of discreteness and projective-quasi-regularity for unital irreducible inclusions of C*-algebras $A\overset{E}{\subset} B$. 
Akin to discrete subfactors, irreducible C*-discrete inclusions are characterized by a standard invariant comprising an action of a UTC and a connected C*-algebra object in its algebraic ind-completion. 
The class of C*-discrete inclusions encompasses all finite Watatani-index, certain semicircular systems, and cores of Cuntz algebras, as well as crossed products by actions of discrete (quantum) groups. 

A main objective of this manuscript is to extend the standard invariant to non-irreducible inclusions of C*-algebras. A prominent class of examples comes from the Morita theory of compact quantum group actions, the relevant inclusions beings $B^\bbG \subset B$, where $B$ is a C*-algebra acted by a compact quantum group and $B^\bbG$ is the fixed point C*-subalgebra. Such inclusions have a well understood categorical description based on Tannaka-Krein duality \cite{MR3121622,MR3420332,MR3426224}. Such an approach to quantum group actions has been very useful in classification problems, e.g. \cite{MR3837600}. In this context, irreducibility corresponds to ergodicity of the action of the compact quantum group. Cores of Cuntz algebra fall into this class of examples.

Another source of motivation for this work comes from {\em factorization homology} with values in C*-categories (cf \cite{2023arXiv230407155H}).
Starting with a unitarily braided unitary tensor category $\cC$, factorization homology gives a functorial construction associating a C*-tensor category $F_\cC(\Sigma)$ to every compact oriented surface $\Sigma$. When $\Sigma$ has non-empty boundary, $F_\cC(\Sigma)$ is described by a C$^*$-algebra object that is generally not connected.

We shall now describe the contents of this article. Section 2 is a preliminary section where we gather some background material, fix notations and conventions. In Section 3 we introduce the {\em coend realization} construction that will serve as main tool in this paper. Given a UTC $\cC$, C*-algebra objects $\bbA$ in $\Vec(\cC^{\op})$ and $\bbD$ in $\Vec(\cC)$, it follows from \cite[\S 4.4]{MR3948170}, the existence of a C*-algebra $C^*_u| \bbA \odot \bbD|$. We show that there is also a minimal completion, which we denote by $\bbA \bowtie \bbD$, and call it the {\em coend realization} of the pair $(\bbA,\bbD)$. Let $\bbOne$ be a tensor unit in $\cC$. We then prove the following.

\begin{thmalpha}[Theorem \ref{thm:condexponthecoendrealization}]\label{thmalpha:condexponthecoendrealization}
    There is a canonical faithful conditional expectation $\bbE: \bbA \bowtie \bbD \to \bbA(1) \otimes \bbD(\bbOne)$.
\end{thmalpha}
\noindent The proof of the theorem relies on the existence of canonical expectations of finite Pimsner-Popa indices $\bbE^{\bbA}_X: \bbA(\overline{X} X) \to \bbA(1)$ and $\bbE^\bbD_X: \bbD(\overline{X} X) \to \bbD(\bbOne)$, which allow for a local comparison between operator norms and Hilbert C*-module norms.

We then consider the above construction to the case when $\bbA$ is the action of a UTC on a unital C*-algebra $A$ with trivial center. This means that $\bbA$ is a unitary tensor functor $\bbA: \cC \to \fgpBim(A)$ from $\cC$ to the UTC of finitely generated projective bimodules of $A$. In this context, we write $A \rtimes \bbD$ for $\bbA \bowtie \bbD$, and call it the crossed product of $A$ by $\bbD$ under the action of $\cC$.
Then, by applying Theorem A, since $\bbA(1) = A$, we obtain the inclusions
\[ A \subset A \otimes \bbD(\bbOne) \overset{\bbE}{\subset} A \rtimes \bbD, \]
where the last inclusion is faithful. It is then easy to see, assuming $\cZ(A) \simeq \bbC$, that the possible descents of $\bbE: A \rtimes \bbD \to A \otimes \bbD(\bbOne)$ to a conditional expectation $E: A \rtimes \bbD \to A$ are parametrized by states on the C$^*$-algebra $\bbD(\bbOne)$.

As mentioned earlier, in the case of a $\rm II_1$-factor $N$, a discrete extension $N \subset M$ can be characterized in two ways; either in terms of the quasi-normalizer or in terms of $\End_{N-N}(L^2(M))$ being a discrete von Neumann algebras. For extensions of C$^*$-algebras, it is yet not known if the two, accordingly adapted to the C$^*$-context, are equivalent.

Taking this into account, in Section 4, we make the distinction between flavors of discreteness for a faithful inclusion $A \overset{E}{\subset} D$. Using the faithful conditional expectation $E$, we can complete $D$ to an $A\-A$-correspondence $\cE$. Discreteness and their variations  are related to whether $D$ or $\cE$ can be reconstructed from $\fgpBim(A)$. The analog concepts in the case of irreducible extensions of type $\rm II_1$-factors all coincide. 

We call a faithful inclusion $A \overset{E}{\subset} D$ discrete if the $*$-subalgebra of $D$ generated by irreducible finitely generated projective $A\-A$-subbimodules of $D$ is dense in the operator norm. 
There is then a category $\CDisc(A)$ whose objects are discrete inclusions $A \overset{E}{\subset}D$ and whose morphisms  $(A \overset{E_1}{\subset}D_1) \to (A \overset{E_2}{\subset}D_2)$ are expectation-preserving ucp maps $D_1 \to D_2$. 
Now the crossed product construction we introduce in Section 4 takes a C*-algebra object $\bbD \in \Vec(\fgpBim(A))$, together with a faithful state $\omega \in \bbD(\bbOne)$, and produces a discrete inclusion $A \subset A \rtimes \bbD$. We then prove the following reconstruction theorem.

\begin{thmalpha}[Theorem \ref{thm: thereconstructiontheorem}]\label{thmalpha: thereconstructiontheorem}
\label{alphathm:thereconstructiontheorem}
    Let $\rCorr_\Omega(\Vec(\fgpBim(A)))$ be the category of C*-algebra objects in  $\fgpBim(A)$ with prescribed faithful states. Morphisms are ucp maps (in the sense of \cite{JP17}, Section 4) preserving the prescribed states. There is an equivalence
    \begin{align*} \rCorr_\Omega(\Vec(\fgpBim(A))) \simeq \CDisc(A) . 
    \end{align*}
\end{thmalpha}
\noindent In other words, a discrete extension of $A$ is classified by a pair $(\bbD,\omega) \in \rCorr_\Omega(\Vec(\fgpBim(A)))$.

Consider, as above, a faithful inclusion $A \overset{E}{\subset} D$, with corresponding C$^*$-algebra object $\bbD$ defined in the algebraic ind-completion of $\fgpBim(A)$. We say that $A \overset{E}{\subset} D$ is {\em projective-quasi-regular}, short PQR, the images of the canonical evaluation maps $K \odot \bbD(K) \to D$ generate, as $K$ ranges over the irreducibles in $\fgpBim(A)$, a dense subspace of the 
 completion $\cE$ of $D$ as an $A \- A$-correspondence. In Section 5 we show that the corresponding relative commutant $A' \cap \End_{\bbC \-A}(\cE) = \End_{A\-A}(\cE)$ of a PQR inclusion $A \overset{E}{\subset} D$ indeed decomposes as a direct product of type $\rm I$ factors, but in principle such a decomposition does not characterize projective quasi-regularity in the non-irreducible case. Indeed, such a decomposition characterizes $\cE$ as an object in $\Hilb(\fgpBim(A))$, i.e., as a unitary ind-object. In that case, we say that $A \overset{E}{\subset}D$ is an {\em ind-inclusion}, and prove the following generalization of Proposition 2.10 in \cite{2023arXiv230505072H}.

\begin{thmalpha}[Theorem \ref{thm:relcommutchar<-}]\label{thmalpha:relcommutchar<-}
    Let $\cE := \overline{D}^{\| \cdot \|_A}$. The inclusion $A \overset{E}{\subset}D$ is an ind-inclusion if and only if there are Hilbert spaces $\{H_i\}_{i \in I}$ such that 
    \[\End_{A\-A}(\cE) \simeq \prod_{i \in I}^{\ell^\infty} \cB(H_i) . \]
\end{thmalpha}

Section 5 is dedicated to examples, starting with a few considerations explaining how the theory of compact quantum group actions on C$^*$-algebras fit into our framework, a discussion that will be no novelty to the expert.

We characterize when a covariance matrix $\eta$ on a C*-algebra $A$ with trivial center leads to an ind-inclusion $A \subset \Phi(\eta)$, the latter being the C*-algebra of the corresponding $A$-valued semi-circular system \cite{Shl99}. 
We also explain how factorization homology gives rise to C*-algebra objects, and therefore to discrete inclusions through Theorem \ref{alphathm:thereconstructiontheorem}, which come with canonical actions of mapping class groups of surfaces.
We mention there are other known families of not necessarily irreducible C*-discrete inclusions arising from Cuntz-Pimsner algebras of dualizable correspondences, which were studied in \cite{2025arXiv250321515H}.

When compared C$^*$-discreteness to W$^*$-discreteness, a central improvement is that the former is defined without making reference to states, as opposed to subfactor discreteness, which carries substantial spacial data compatible with the unique trace on a $\rm{II}_1$-factor. 
It is then meaningful to ask to what extent are these notions of discreteness compatible. 
We address this question in the case where the base C*-algebra $A$ is strongly dense in $N$, and the discrete extensions $(A\subset B, E_A)$ and $((N, \tr))\subset M, E_B)$ are related by a \emph{compatible action} of $\cC$, which is compatible with the spacial data carried by the tracial state on $N$ and the associated spherical bimodules.
In this context, we prove that the double commutant of an irreducible C*-discrete inclusion gives an extremal irreducible discrete subfactor: 
\begin{thmalpha}[{Theorem~\ref{thm:vNvsC*Disc}}]\label{thmalpha:vNvsC*Disc}
Given a unitary tensor category $\cC$ and compatible actions $F:\cC\to\fgpBim(A)$ and $F'':\cC\to \spbfBim(N).$
Then, for any C*-discrete extension $A\overset{E}{\subset}B$ supported on $F[\cC]$, there is a corresponding extremal discrete subfactor $N\overset{E''}{\subset} M$ obtained as a double commutant in a compatible representation. 
\end{thmalpha}
\noindent As a consequence, in Corollary \ref{Prop:Galois} we obtain a partial comparison between the lattices of intermediate subalgebras $A\subset D\subset B$ and $N\subset P \subset M,$ which restricts to a bijection when considering only those C*-discrete $A\subset D$.

\bigskip
\paragraph{Acknowledgements}

LH was supported by the Engineering and Physical Sciences Research Council (EP/X026647/1). RHP was partially supported by an NSERC Discovery Grant. 

\bigskip
{\em Open Access:} For the purpose of Open Access, the authors have applied a CC BY public copyright licence to any Author Accepted Manuscript (AAM) version arising from this submission

\tableofcontents

\section{Preliminaries}
\label{sec:preliminaries}
In this section we introduce notation for C$^*$-correspondences, unitary tensor categories (UTCs), and related notions that will be used throughout the paper. 
For more detailed descriptions, we refer the reader to \cite{JP17, MR3948170,MR4419534, 2023arXiv230505072H} and the references therein. Most of our notation regarding C$^*$-tensor categories and C$^*$-correspondences will be imported from \cite{2023arXiv230505072H}. 

There will be several notions of tensor products present in the article. We write $\odot$ for the algebraic tensor product. The symbol $\otimes$ will denote either the minimal tensor product of C$^*$-algebras or the external tensor product of Hilbert C$^*$-modules. The Connes fusion or relative tensor product of Hilbert C$^*$-bimodules over a C$^*$-algebra $A$ will be denoted by $\underset{A}{\boxtimes}$, or just $\boxtimes$ if the C$^*$-algebra $A$ is clear from the context.

\subsection{C*-correspondences}

Let $A$ and $B$ be unital C$^*$-algebras. A right $A \- B$-correspondence is a right Hilbert C$^*$-module $K$ over $B$ together with a unital $*$-homomorphism $\pi: A \to \cL_B(K)$, where $\cL_B(K)$ denotes the C$^*$-algebra of adjointable $B$-linear operators on $K$. When $\pi(A) (K)$ is a dense subspace of $K$, we say that $K$ is a {\em non-degenerate right $A \- B$-correspondence}.

\begin{defn}
Given unital C$^*$-algebras $A$ and $B$, let $\rCorr_{A \- B}$ be the category having non-degenerate right $A \- B$-correspondences as objects and adjointable $A \- B$-bilinear maps as morphisms. For $K_1,K_2 \in \rCorr_{A \- B}$, let $\rCorr_{A \- B}(K_1,K_2)$ denote the space of morphisms. Observe that for $K \in \rCorr_{\bbC \- A}$, $\rCorr_{\bbC \- A}(K) = \cL_A(K)$.
\end{defn}

\begin{defn}[\cite{MR2085108}]
An object $K \in \rCorr_{A\-A}$, the right $A$-valued inner product of which we denote by $\tensor[]{\langle \cdot , \cdot \rangle}{_A}$, is {\em bi-Hilbertian} if it admits a left $A$-valued inner product $\tensor[_A]{\langle \cdot , \cdot \rangle}{}$ such that the topology it induces on $K$ coincides with the topology induced by $\tensor[]{\langle \cdot , \cdot \rangle}{_A}$.
\end{defn}

\begin{defn}
A bi-Hilbertian right $A$-$A$- C$^*$-correspondence $K$ is said to be finitely generated projective if it is finitely generated projective as a right $A$-module. For a unital C$^*$-algebra $A$, we denote by $\fgpBim(A)$ the full subcategory of $\rCorr_{A\-A}$ consisting of finitely generated projective bi-Hilbertian correspondences $K \in \rCorr_{A\-A}$.
\end{defn}

The discussion above can be extended to define a 2-C$^*$-category $\rCorr$.  See \cite{2023arXiv230505072H}, Section 1, for more details.

\subsection{Unitary tensor categories}

A linear category $\cC$ is a locally small category such that for every pair $X,Y \in \cC$, the space $\cC(X,Y)$ of morphisms from $X$ to $Y$ is equipped with a structure of a vector space over $\bbC$. The model example is the category $\Mod(A)$ of left $A$-modules for an algebra $A$ over $\bbC$, where the morphisms between two modules are linear maps between them intertwining the corresponding actions.
A dagger, or {\em $*$-category}, is a linear category $\cC$ equipped with natural involutions $(-)^*: \cC(X,Y) \to \cC(Y,X)$. Naturality means that $(-)^*$ can be seen as a funtcor
\[ \cC \to \cC^{\op} , \]
acting as the identity function on the set of objects. In the category $\Hilb_{fd}$ of finite dimensional Hilbert spaces, for instance, the standard $*$-structure is given by the adjoint operation on the spaces of bounded linear maps.

\begin{defn}
Let $\cC$ be a $*$-category, and suppose that $\cC$ admits finite direct sums. Then $\cC$ is a C$^*$-category if $\cC(X) := \cC(X,X)$ is a C$^*$-algebra for every $X \in \cC$.
\label{def:Cstarcat}
\end{defn}

\begin{defn}
A C$^*$-tensor category is a C$^*$-category $\cC$ equipped with a tensor structure $\boxtimes: \cC \times \cC \to \cC$ such that
\[ ((-) \boxtimes (-))^* = (-)^* \boxtimes (-)^* \]
on morphisms.
\label{def:Cstartensorcat}
\end{defn}

We shall make use of the following terminology. For tensor categories $\cC_1$ and $\cC_2$, a tensor functor $F:\cC_1 \to \cC_2$ is understood to have natural isomorphisms as tensor structure. A lax functor $F:\cC_1 \to \cC_2$ is understood to be a functor equipped with associators and unitors, but they are not required to be natural isomorphisms.

\begin{defn}\label{def:unitarytensorfunctor}
    Let $\cC_1$ and $\cC_2$ be C$^*$-tensor categories. A tensor functor $F:\cC_1 \to \cC_2$ is a {\bf unitary tensor functor} if it is $*$-preserving and the tensor structure consist of unitaries. 
\end{defn}

We will assume throughout the paper all tensor categories of interest to have been strictified.

Examples of C$^*$-tensor categories include: the category $\Rep(\bbG)$ of unitary representations of a  locally compact (quantum) group $\bbG$, the category $\Bim(N)$ of bimodules over a von Neumann algebra $N$ and the category $\rCorr_{A\-A}$ of non-degenerate right C$^*$-correspondences over a C$^*$-algebra $A$. 

We will find latter in this paper unitary functors $\cC \to \rCorr_{A\-A}$ that are almost unitary tensor functors, Definition \ref{def:unitarytensorfunctor} being broken due to the associators not being adjointable morphisms. 

\begin{defn}
    A functor $F:\cC \to \rCorr_{A\-A}$ is a {\bf weak unitary tensor functor} if it is unitary and if it is equipped with isometries
    \[F^2 = \{ F^2_{X,Y}: F(X) \boxtimes F(Y) \to F(X \boxtimes Y) \ | \ X,Y \in \cC\}, \]
    not necessarily adjointable, satisfying pentagon coherence.
\end{defn}

\begin{rem}
    Weak unitary tensor functors are the key to Tannaka-Krein duality for compact quantum group actions \cite{MR3426224}. 
\end{rem}

We wish to introduce now the concept of duality in C$^*$-tensor categories. As a motivating example, let us fix a unital C$^*$-algebra $A$, and consider the subcategory $\fgpBim(A)$ of $\rCorr_{A\-A}$ consisting of bi-Hilbertian, finitely generated projective correspondences over $A$. If $K \in \fgpBim(A)$, it admits a right Pimsner-Popa basis, consisting of a finite set $\{\xi_i\}_i \subset K$ such that
\[ \eta = \sum_i \xi_i \triangleleft \tensor[]{\langle \eta,\xi_i \rangle}{_A}  , \ \forall \ \eta \in K . \]
In particular, $\{\xi_i\}$ is a projective basis for $K$ over $A$, so `$K$ has finite dimension' over $A$.  Because $K$ also has a left $A$-valued inner product, the conjugate space $\bar{K}$ can also be equipped with the structure of a right $A$-correspondence, and moreover $\bar{K} \in \fgpBim(A)$. Denoting by $\underset{A}{\boxtimes}$ the fusion of correspondences over $A$, it follows that
\[ \sum_i  \bar{\xi_i} \underset{A}{\boxtimes} \xi_i \in \bar{K} \underset{A}{\boxtimes} K \]
is an $A$-central vector, providing a canonical $A\-A$-bimodular linear map $R_K: A \to \bar{K} \underset{A}{\boxtimes} K$.

Given a tensor category $\cC$, we will denote a choice of tensor unit for $\cC$ by $\bbOne$
\begin{defn}
Let $\cC$ be a C$^*$-tensor category. An object $X \in \cC$ is {\em dualizable} if there exist: an object $\bar{X} \in \cC$, morphisms $R_X: \bbOne \to \bar{X} \boxtimes X$, $\bar{R}_X: \bbOne \to X \boxtimes \bar{X}$ such that
\begin{center}
\begin{tikzcd}
X \arrow[rr, "\id_X \boxtimes R_X"] \arrow[rd, "\id_X"'] &   & X \boxtimes\bar{X} \boxtimes X \arrow[ld, "\bar{R}_X^* \boxtimes \id_X"] &  & \bar{X} \arrow[rr, "\id_{\bar{X}} \boxtimes \bar{R}_X"] \arrow[rd, "\id_{\bar{X}}"'] &         & \bar{X} \boxtimes X \boxtimes \bar{X} \arrow[ld, "R_X^* \boxtimes \id_{\bar{X}}"] \\
                                                         & X &                                                                          &  &                                                                                      & \bar{X} &                                                                                  
\end{tikzcd}
\end{center}
commute.
\end{defn}

\begin{defn}
A C$^*$-tensor category $\cC$ is rigid if every object in $\cC$ is dualizable. 
\end{defn}

A pair $(R_X, \bar{R}_X)$ is said to be a {\em solution of the conjugate equations} for $X$.

\begin{defn}
Let $\cC$ be a rigid C$^*$-tensor category, and let $X$ be an object of $\cC$. The {\em intrinsic dimension} $d_X$ of $X$ in $\cC$ is defined by
\[ d_X := \underset{(R_X, \bar{R}_X)}{\min} \|R_X\| \cdot \| \bar{R}_X\| , \]
where  $(R_X, \bar{R}_X)$ runs through the solutions of the conjugate equations for $X$.
\end{defn}

The function defined on the set of objects of $\cC$ given by $X \mapsto d_X$ is a {\em dimension function}. In particular, $d_{\bbOne} = 1$, $d_{X \boxtimes Y} = d_X d_Y$ and $d_{X \oplus Y} = d_X + d_Y$.

\begin{defn}
A rigid C$^*$-tensor category is a {\em unitary tensor category}, short UTC, if its unit object $\bbOne$ is simple, meaning $\cC(\bbOne) \simeq \bbC$.
\end{defn}

\begin{rem}
It can be shown that, in a unitary tensor category, every morphism space is finite dimensional. In particular, the endomorphism algebras are finite dimensional C$^*$-algebras and are therefore semi-simple. This implies $\cC$ is semi-simple: a collection $\Irr(\cC) \subset \cC$ of objects can be chosen with the following properties:
\begin{itemize}
\item $\cC(X) \simeq \bbC$ for all $X$  in $\Irr(\cC)$ (we say that $X$ is irreducible, or simple);
\item for $X,Y \in \Irr(\cC)$, with $X \neq Y$, $\cC(X,Y) = 0$;
\item every object in $\cC$ is isomorphic to a finite direct sum of objects in $\Irr(\cC)$.
\end{itemize}
\end{rem}

\begin{ex}
Let $A$ be a unital C$^*$-algebra. Then the subcategory of dualizable objects in $\rCorr_{A\-A}$ is exaclty $\fgpBim(A)$. Thus the latter is a rigid C$^*$-tensor category. The tensor unit of this category is $A$ with its canonical structure of an $A\-A$-correspondence. We have that the endomorphism algebra of $A$ in $\fgpBim(A)$ is the center $\cZ(A)$ of $A$. It follows that $\fgpBim(A)$ is a unitary tensor category if and only if $\cZ(A) \simeq \bbC$. We shall write $\Irr(A):= \Irr(\fgpBim(A))$ for simplicity.
\end{ex}

\subsection{ind-completions and C*-algebra objects}
\label{subsec:indcomp}

The treatment of infinite index inclusions will require the use of objects that are not dualizable, but that can be decomposed into dualizable objects. We shall now introduce the formalism needed to describe this situation. The reader is refered to \cite{JP17} for more details.

\begin{defn}
Let $\cC$ be a unitary tensor category. The {\em algebraic ind-completion} $\Vec(\cC)$ of $\cC$ is the category of functors $\cC^{\op} \to \Vec$. Morphisms in this category are linear natural transformations.
\end{defn}

To specify an object $V \in \Vec(\cC)$, it suffices, up to isomorphism, to specify its values, or {\em fibers}, at the objects in $\Irr(\cC)$, i.e., a family $\{V(X)\}_{X \in \Irr(\cC)}$ of vector spaces indexed by $\Irr(\cC)$. Another useful perspective is that of understanding the object $V = \{V_X\}_{X \in \Irr(\cC)} \in \Vec(\cC)$ as a  formal direct sum
\[ V = \bigoplus_{X \in \Irr(\cC)} X \odot V(X) \]
of objects in $\cC$, where the irreducible object $X$ appears $\dim V(X)$ times as a direct summand. We say that $V(X)$ is the {\em multiplicity space} of $X$ in $V$. Moreover, given $V_1 = \{V_1(X)\}_X$ and $V_2=\{V_2(X)\}_X$ in $\Vec(\cC)$, it follows that
\[ \Vec(\cC)(V_1,W_1) \simeq \prod_{X \in \Irr(\cC)} \Hom_{\Vec}(V_1(X),V_2(X)) , \]
where $\Hom_\Vec(\cdot,\cdot)$ denotes the space of linear maps.

The category $\cC$ embedds into $\Vec(\cC)$ as a linear full subcategory, and the tensor structure $\boxtimes$ in $\cC$ can be extended to a tensor structure $\boxdot$ on $\Vec(\cC)$: for $V$ and $W$ as above,
\[ (V_1 \boxdot V_2)(X) : = \bigoplus_{Y,Z \in \Irr(\cC)} \cC(X, Y \boxtimes Z) \odot V_1(Y) \odot V_2(Z) . \]

\begin{defn}
The {\em unitary ind-completion} of a unitary tensor category $\cC$ is the category $\Hilb(\cC)$ with
\begin{itemize}
\item objects: functors $\cC^{\op} \to \Hilb$;
\item morphisms: uniformly bounded linear natural transformations. That is, given $H_1, H_2 \in \Hilb(\cC)$, a morphism $\eta \in \Hilb(\cC)(H_1,H_2)$ is a natural family $\{ \eta_X \in \cL(H_1(X),H_2(X))\}_{X \in \cC}$ such that
\[ \underset{X \in \cC}{\sup} \| \eta_X \| < \infty . \]
\end{itemize}
\end{defn}

An object $H \in \Hilb(\cC)$ is determined by a family $\{H(X)\}_{X \in \Irr(\cC)}$ of Hilbert spaces. We have
\[ \Hilb(\cC)(H_1,H_2) \simeq \prod_{X \in \Irr(\cC)}^{\ell^\infty} \cL(H_1(X),H_2(X)) . \]

The category $\Hilb(\cC)$ is a C$^*$-category (actually a W$^*$-category), and $\cC$ is a full C$^*$-subcategory of it. The tensor structure of $\cC$ extends to $\Hilb(\cC)$; seeing $\cC(X,Y  \boxtimes Z)$ as a Hilbert space,
\[ (H_1 \boxtimes H_2)(X) = \bigoplus_{Y,Z \in \Irr(\cC)}^{\ell^2} \cC(X, Y \boxtimes Z) \otimes H_1(Y) \otimes H_2(Z) , \]
where on the $(Y,Z)$-direct summand the inner-product is scaled by a factor of $(d_Y d_Z)^{-1}$. The purpose of this renormalization is to have a clean interpretation of the induced inner product by means of graphical calculus (see \cite{JP17}, Section 2).

\medskip

Since $\Vec(\cC)$ is a tensor category, we can consider algebra objects in it. An algebra object  in $\Vec(\cC)$ consists of a functor $\bbA: \cC^{\op} \to \Vec$ together with a lax-natural transformation 
\[\bbA^2 = \{ \bbA^2_{X,Y} : \bbA(X) \odot \bbA(Y) \to \bbA(X \boxtimes Y) \ | \ X,Y \in \cC\} ,\]
satisfying coherence with respect to the pentagon diagram associated to the orderings of triple tensor products.

\begin{defn}
A $*$-structure on an algebra object $\bbA$ is a conjugate linear natural transformation 
\[ j^\bbA = \{ j^\bbA_X: \bbA(X) \to \bbA(\bar{X})\}_{X \in \cC} , \]
satisfying
\begin{itemize}
\item involution: $j^\bbA_{\bar{X}} \circ j^\bbA_X = \id_{\bbA(X)}$, under the identification $\bar{\bar{X}} \simeq X$;
\item unitality: $j^\bbA_{1_\cC} = \id_{\bbA(1_\cC)}$, under the identification $\bar{\bbOne} \simeq \bbOne$;
\item monoidality: the diagram
\begin{center}
\begin{tikzcd}
\bbA(X) \odot \bbA(Y) \arrow[rr, "j^\bbA_X \odot j^\bbA_Y"] \arrow[dd, "\bbA^2"'] &  & \bbA(\bar{X}) \odot \bbA(\bar{Y}) \arrow[dd, "\bbA^2"]              \\
                                                                                  &  &                                                                     \\
\bbA(X \boxtimes Y) \arrow[rr, "j^\bbA_{X \boxtimes Y}"']                         &  & \bbA(\bar{X}\boxtimes \bar{Y}) \simeq \bbA(\overline{Y\boxtimes X})
\end{tikzcd}
\end{center}
commutes.
\end{itemize}
\end{defn}

If $\bbA = (\bbA, \bbA^2)$ is an algebra object in $\Vec(\cC)$, then the vector spaces $\bbA(\bar{X} \boxtimes X)$ become algebras, with multiplication given by
\[ \bbA(\bar{X} \boxtimes X) \odot \bbA(\bar{X} \boxtimes X) \overset{\bbA^2}{\to} \bbA( \bar{X} \boxtimes X \boxtimes \bar{X} \boxtimes X) \overset{\bbA(\id_{\bar{X}} {\boxtimes \bar{R}_X^* \boxtimes \id_X)}}{\to} \bbA(\bar{X} \boxtimes X) . \]
If $j^\bbA$ is a $*$-structure on $\bbA$, then $\bbA(\bar{X} \boxtimes X)$ becomes $*$-algebra, with involution given by $j^\bbA_{\bar{X} \boxtimes X}$.

\begin{defn}
Let $\bbA = (\bbA,\bbA^2, j^\bbA)$ be a $*$-algebra in $\Vec(\cC)$. We say that $\bbA$ is a C$^*$-algebra object if, for each object $X$, the $*$-algebra $\bbA(\bar{X} \boxtimes X)$ is a C$^*$-algebra.
\end{defn}

\begin{rem}
    C*-algebra objects can be equivalently defined in terms of C*-module categories as follows: 
    Given a $*$-algebra object $\bbA$ as above, the category $\cM_\bbA$ generated under finite direct sums and idempotent completion by objects of the form $\bbA \odot X$, with $X \in \cC$, and with morphisms 
    \[ \cM_\bbA(\bbA \odot X, \bbA \odot Y) := \Vec(\cC)(X, \bbA \odot Y) \simeq \bbA(X \otimes \bar{Y}) \]
    has a right $\cC$-module structure. $\bbA$ is a C$^*$-algebra object if and only if $\cM_\bbA$ is a $\cC$-module C$^*$-category \cite{JP17}. In particular, if $\bbA$ is a C$^*$-algebra object, $\bbA(X)$ has a Banach space structure for every $X \in \cC$.
\end{rem}

\begin{defn}[\cite{JP17}, Definition 4.20]
Let $\bbD_1$ and $\bbD_2$ be C$^*$-algebra objects in $\Vec(\cC)$.  A natural transformation $\theta: \bbD_1 \to \bbD_2$ is called a $*$-natural transformation if 
\[ j^{\bbD_2} \circ \theta = \theta \circ j^{\bbD_1} . \] 
A $*$-natural transformation $\theta: \bbD_1 \to \bbD_2$ is called a {\em ucp map} if, for all $X \in \cC$, $\theta: \bbD_1(\bar{X} \boxtimes X) \to \bbD_2(\bar{X} \boxtimes X)$ is a ucp map.
\label{defn:categoricalucpmaps}
\end{defn} 

\section{Reduced Realization of C*-algebra objects }
\label{sec:generalrealization}
In this section we introduce a coend-type construction for pairs of C*-algebra objects that is later shown to generalize crossed-products in the spirit of \cite{MR3948170, 2023arXiv230505072H}. 

\subsection{Intrinsic characterization of internal C*-algebra objects}

In this subsection we recast the results in \cite[Appendix A]{2022arXiv220506663H} and  \cite[\S4.2]{JP17}  concerning Pimsner-Popa inequaltities that arise naturally when considering C$^*$-algebra objects in UTC's. Later, these results will be our main tools for producing faithful conditional expectations.

Let $\cC$ be a UTC and $(\bbD, \bbD^2, j^\bbD)$ be a C*-algebra object in $\Vec(\cC)$. Recall in particular that each fiber $\bbD(X)$ has a canonical Banach space structure, the norm of which we denote simply by $\| \cdot \|$, and that $\bbD(\overline{X} \boxtimes X)$ is a C*-algebra, for all $X \in \cC$.

\begin{defn}\label{defi:groundC*algebra}
    Given a C*-algebra object $\bbD\in \Vec(\cC)$, we call $\bbD(\bbOne)$ the ground C*-algebra of $\bbD$.     
    Furthermore, for each $X \in \cC$, fix a standard solution to the conjugate equations $(R_X,\overline{R}_X)$ for the pair $(X, \overline{X})$. Denoting by $d_X$ the dimension of $X$ in $\cC$, we define the map
\[
    \bbE^\bbD_X := d_X^{- 1}  \bbD(R_X): \bbD(\overline{X} \boxtimes X)\to \bbD(\bbOne). 
\]
\end{defn}

\begin{lem}[\cite{JP17}, Corollary 4.2]
    The map $\bbE^\bbD_X: \bbD(\overline{X} \boxtimes X) \to \bbD(\bbOne)$ is a conditional expectation. For all $T \in \bbD(\overline{X} \boxtimes X)_+$ we have that 
    $$ \| \bbE^\bbD_X(T) \| \leq \|T\| \leq d_X^2 \| \bbE^\bbD_X(T) \| . $$
    That is, the conditional expectation $\bbE_X^\bbD$ has finite Pimsner-Popa index.
    \label{lemma:PPindexforcategoricalcondexp}
\end{lem}

Lemma \ref{lemma:PPindexforcategoricalcondexp} implies  that $\bbE_X^\bbD$ is a faithful conditional expectation. Let $X \in \cC$. Then
\begin{align*}
    \bbD^2_{X, \bbOne} &: \bbD(X) \odot \bbD(\bbOne) \to \bbD(X \boxtimes \bbOne) \simeq \bbD(X) \\
    \bbD^2_{\bbOne,X} &: \bbD(\bbOne) \odot \bbD(X) \to \bbD(\bbOne \boxtimes X) \simeq \bbD(X)
\end{align*}
gives to $\bbD(X)$ the structure of a $\bbD(\bbOne)$-bimodule. We claim that this structure can be further enhanced to that of a right C*-correspondence, allowing us to see $\bbD(X)$ as an object in $\rCorr_{\bbD(\bbOne) \- \bbD(\bbOne)}$.

For a proof of the following lemma, see \cite[\S4.2]{JP17}.

\begin{lem}
    For each $X \in \cC$, the composition
    \begin{center}
        \begin{tikzcd}
\bbD(X) \odot \bbD(X) \arrow[rr, "j^\bbD_X \odot \id"] &  & \bbD(\overline{X}) \odot \bbD(X) \arrow[rr, "\bbD^2_{\overline{X},X}"] &  & \bbD(\overline{X} \boxtimes X) \arrow[rr, "\bbE^\bbD_X"] &  & \bbD(\bbOne)
\end{tikzcd}
    \end{center}
    is an $\bbD(\bbOne)$ valued inner product on $\bbD(X)$. For $\xi,\eta \in \bbD(X)$, we shall write
    \[ \langle \xi, \eta \rangle_{\bbD(\bbOne)} : = \bbE^\bbD_X \left( \bbD^2_{\overline{X},X}(j^\bbD(\xi) \odot \eta ) \right) . \]
\end{lem}

For $\xi \in \bbD(X)$, let us write $\| \xi \|_{\bbD(\bbOne)}$ for the norm of $\xi$ in $ \tensor[_{\bbD(\bbOne)}]{\bbD(X)}{_{\bbD(\bbOne)}}$, that is, the norm coming from the $\bbD(\bbOne)$-valued inner product.  Note that, for $\xi \in \bbD(\bbOne)$, $\| \xi \| = \|\xi \|_{\bbD(\bbOne)}$. 

\begin{lem}
    For every $\xi \in \bbD(X)$,
\[ \| \xi \|^2 = \| \bbD^2_{\overline{X},X} (j^\bbD_X(\xi) \odot \xi) \| . \]
\end{lem}
\begin{proof}
    In terms of the $\cC$-module C$^*$-category $\cM_\bbD$ associated to $\bbD$ and the identification $\Psi: \bbD(X) \simeq \cM_\bbD(\bbD \odot X, \bbD)$, the right hand side is just the norm of $\Psi(\xi)^* \circ \Psi(\xi)$.  Since $\cM_\bbD$ is a C*-category,
    \[ \| \Psi(\xi)^* \circ \Psi(\xi) \| = \| \Psi(\xi) \|^2 = \|\xi \|^2. \]
\end{proof}

From Proposition 4.16 in \cite{JP17}, it follows that the norms $\| \cdot \|$ and $\| \cdot \|_{\bbD(\bbOne)}$ in the vector space $\bbD(X)$ are equivalent. The proof of this statement will be important for the coend realization of Section 3, given us the motive to write it down here.

\begin{cor}
    Let $\cC$ be a UTC, $\bbD$ a C*-algebra object in $\Vec(\cC)$, and let $X \in \cC$. For every $\xi \in \bbD(X)$, it holds 
    \[ \| \xi \|_{\bbD(\bbOne)} \leq \| \xi \| \leq d_X \| \xi \|_{\bbD(\bbOne)} . \]
    In particular, $\bbD(X)$ is complete with respect to the $\bbD(\bbOne)$-Hilbert C$^*$-module norm, i.e., $\bbD(X) \in \rCorr_{\bbD(\bbOne) \- \bbD(\bbOne)}$.
\end{cor}

\begin{proof}
First, observe that
\begin{align*}
    \| \xi \|_{\bbD(\bbOne)}^2 & = \| \bbE^\bbD_X \left( \bbD^2_{\overline{X},X}(j^\bbD_X(\xi) \odot \xi ) \right) \| \\
    & \leq \| \bbD^2_{\overline{X},X}(j^\bbD_X(\xi) \odot \xi ) \| \\
    & = \| \xi \|^2 .
\end{align*}
The last equality follows from the previous lemma. On the other hand, by Lemma \ref{lemma:PPindexforcategoricalcondexp}, we have
\begin{align*}
    \| \xi \|^2 & = \| \bbD^2_{\overline{X},X}(j^\bbD_X(\xi) \odot \xi ) \| \leq d_X^2 \| \bbE_X ( \bbD^2_{\overline{X},X}(j^\bbD_X(\xi) \odot \xi )) \| = d_X^2 \|\xi\|_{\bbD(\bbOne)}^2 .
\end{align*}

\end{proof}

Thus, if $\bbD$ is a C*-algebra object in a UTC $\cC$, then $\bbD$ has the structure of a lax tensor functor $\bbD: \cC^{\op} \to \rCorr_{\bbD(\bbOne)\-\bbD(\bbOne)}$. It can be deduced moreover that the components of the lax structure $\bbD^2$ are isometries, possibly non-adjointable, i.e., $\bbD^2$ is a  {\em weak unitary tensor functor} as introduced in the previous Section.

\begin{lem}
    Let $\bbD$ be a C$^*$-algebra object in $\Vec(\cC)$. Then $\bbD^2$ induces isometric linear maps 
    \[ \bbD(X) \underset{\bbD(\bbOne)}{\boxtimes} \bbD(Y) \to \bbD(X \boxtimes Y) \]
    for $X,Y \in \cC$.
\end{lem}
\begin{proof}
    Fix $X,Y \in \cC$. It follows from the axioms of $\bbD^2$ that $\bbD^2_{X,Y}: \bbD(X) \odot \bbD(Y)$ is $\bbD(\bbOne)$-balanced. Using the right $\cC$-module C$^*$-category $\cM_\bbD$ associated to $\bbD$, given $f \in \cM_\bbD(\bbD \odot X, \bbD)$ and $g \in \cM_{\bbD}(\bbD \odot Y, \bbD)$, we have
    \[ \bbD^2_{X,Y}( f \odot g) = g  (f \odot \id_Y): \bbD \odot (X \boxtimes Y) \simeq (\bbD \odot X) \odot Y \to \bbD , \]
    and that
    \begin{align*}
        \langle \bbD^2_{X,Y}( f \odot g), \bbD^2_{X,Y}( f \odot g) \rangle_{\bbD(\bbOne)} &= g  (f \odot\id_Y) (f^* \odot \id_Y) g^* \\
        & = \langle g , \langle f,f \rangle_{\bbD(\bbOne)} \triangleright g \rangle_{\bbD(\bbOne)} \\ 
        & = \langle f \boxtimes g, f \boxtimes g \rangle_{\bbD(\bbOne)} . 
    \end{align*}
\end{proof}

We summarize our conclusions in the following Corollary.

\begin{cor}
    Let $\cC$ be a UTC and let $\bbD$ be a C$^*$-algebra object in $\Vec(\cC)$. Then the functor $\bbD: \cC^{\op} \to \Vec$ can be endowed with a canonical structure of a weak unitary tensor functor from $\cC^{\op}$ to $\rCorr_{\bbD(\bbOne) \- \bbD(\bbOne)}$.
    \label{cor:intrinsiccharC*algobjs}
\end{cor}

To finish the preliminary section, we have one final observation. This appears in Appendix A of \cite{2022arXiv220506663H}.
\begin{lem}
    If $\bbD$ is a C*-algebra object as above, for each $X \in \cC$, $\bbD(X)$ has a canonical structure of right Hilbert $\bbD(\overline{X} \boxtimes X)$-module.
    \label{lem:canonicalD(barXX)-modulestructure}
\end{lem} 
\begin{proof}
    The action is defined, for $a \in \bbD(X)$ and $x \in \bbD(\overline{X} \boxtimes X)$, by
    \[ a \triangleleft x := \bbD(R_X \boxtimes \id_X) \left(\bbD^2_{X, \overline{X} \boxtimes X}( a \odot x) \right) . \]
   The $\bbD(\overline{X} \boxtimes X)$-valued inner product is given by
   \[ \langle a,a' \rangle_{\bbD(\bar{X} \boxtimes X)} : = \bbD^2_{\overline{X},X}(j^\bbD(a) \odot a') , \]
   which is shown to be $\bbD(\overline{X} \boxtimes X)$-linear by means of the following computation:
   \begin{align*}
       \langle a, a' \triangleleft x \rangle _{\bbD(\bar{X} \boxtimes X)} & = \bbD^2_{\overline{X},X}\left(j^\bbD(a) \odot \bbD(R_X \boxtimes \id_X) (\bbD^2_{X,\overline{X} \boxtimes X}(a' \odot x))\right) \\
       & = \bbD(\id_{\overline{X}} \boxtimes R_X \boxtimes \id_X) \bbD^2_{\overline{X},X \boxtimes \overline{X} \boxtimes X} \left( j^\bbD(a) \odot \bbD^2_{X,\overline{X} \boxtimes X}(a' \odot x) \right) \\
       & = \bbD(\id_{\overline{X}} \boxtimes R_X \boxtimes \id_X) \bbD^2_{\overline{X} \boxtimes X,{\overline{X} \boxtimes X}}( \bbD^2_{\overline{X},X}(j^\bbD(a) \odot a') \odot x) \\
       & = \langle a,a'\rangle_{\bbD(\bar{X} \boxtimes X)} \cdot x .
   \end{align*}

   The right Hilbert $\bbD(\bbOne)$-module structure of $\bbD(X)$ is induced from the Hilbert $\bbD(\overline{X} \boxtimes X)$-module structure together with the conditional expecation $\bbE^\bbD_X: \bbD(\overline{X} \boxtimes X) \to \bbD(\bbOne)$, meaning
   \[ \<x,y\>_{\bbD(\bbOne)} = \bbE_X^\bbD(\<x,y\>_{\bbD(\bar{X} \boxtimes X)}) .\] 
   Thus the Banach space structure of $\bbD(X)$ also coincides with the one coming from the right Hilbert $\bbD(\overline{X} \boxtimes X)$-module structure. From this it follows that the topology on $\bbD(X)$ induced by the $\bbD(\bar{X} \boxtimes X)$-valued inner product is the same as the one induced by the $\bbD(\bbOne)$-valued inner product, due to finiteness of the index of $\bbE_X^\bbD$. In particular, $\bbD(X)$ is closed for the former topology.
\end{proof}

\subsection{Coend realization} 

Given an arbitrary UTC $\cC$ and C*-algebra objects $\bbA$ and $\bbB$ in $\Vec(\cC^{\op})$ and $\Vec(\cC)$, respectively, it was shown in \cite{MR3948170} that the $*$-algebra 
\[ |\bbA \times \bbB| := \bigoplus_{X \in \Irr(\cC)} \bbA(X) \odot \bbB(X) \]
has a universal C*-completion.
We show in this section that there is a canonical minimal C*-completion $\bbA \bowtie \bbB$ of $|\bbA \times \bbB|$, analogous to regular representations, featuring a faithful conditional expectation $\bbA \bowtie \bbB \to \bbA(\bbOne) \otimes \bbB(\bbOne)$. This is analogous to the comparison between universal and reduced group C$^*$-algebras.

For $X \in \cC$, as in Definition \ref{defi:groundC*algebra} we write $\bbE^\bbA_X$ and $\bbE^\bbB_X$ for the conditional expectations $\bbA(\overline{X} \boxtimes X) \to \bbA(\bbOne)$ and $\bbB(\overline{X} \boxtimes X) \to \bbB(\bbOne)$, respectively. Let $\bbA(X) \otimes \bbB(X)$ be the exterior tensor product of the $\bbA(\bbOne)$-$\bbA(\bbOne)$-correspondence $\bbA(X)$ with the $\bbB(\bbOne)$-$\bbB(\bbOne)$-correspondence $\bbB(X)$, and define 
\[ \cE := \bigoplus_{X \in \Irr(\cC)}^{\ell^2} \bbA(X) \otimes \bbB(X) \]
to be the completion of $\bigoplus_{X \in \Irr(\cC)} \bbA(X) \otimes \bbB(X)$ as a $(\bbA(\bbOne) \otimes \bbB(\bbOne))$-$(\bbA(\bbOne) \otimes \bbB(\bbOne))$-correspondence. 

Given irreducible objects $X_0,X_1,X_2 \in \cC$, there is a Hilbert space structure on $\cC(X_0, X_1 \boxtimes X_2)$ given by $\langle x,  y \rangle := x^* y$. Let us fix, for each triple $(X_0,X_1,X_2)$ of irreducible objects, a choice of an orthonormal basis $O(X_0,X_1 \boxtimes X_2)$ for this Hilbert space.

\begin{lem}
    Let $\cC$ be a UTC and $\bbD$ a C*-algebra object in $\Vec(\cC)$. For $X_0,X_1,X_2 \in \Irr(\cC)$, given $\xi_0 \in \bbD(X_0)$, the linear map $\xi_0 \triangleright -: \bbD(X_1) \to \bbD(X_2)$ given by
    $$\xi_0 \triangleright \xi_1 := \sum_{v \in O(X_2,X_0  \boxtimes X_1)} \bbD(v) \left( \bbD^2_{X_0,X_1}(\xi_0 \odot \xi_1) \right)$$
    is an adjointable map of $\bbD(\bbOne)$-Hilbert C*-modules. Its adjoint is then given by $j^\bbD_{X_0}(\xi_0) \triangleright (-)$.
\end{lem}

\begin{proof}
    Given $X \in \cC$, the definition of the $\bbD(\bbOne)$-valued inner product on $\bbD(X)$ is
    $$\< \xi, \xi' \>_{\bbD(\bbOne)} := \bbD(R_X) \left( \bbD^2_{\overline{X},X} ( j^\bbD(\xi) \odot \xi') \right) . $$
    Then, for $\xi_i \in \bbD(X_i)$, $i \in \{0,1,2\}$,
    
    \begin{align*}
        \< \xi_0 \triangleright \xi_1, \xi_2 \>_{\bbD(\bbOne)} & = \sum_{v \in O(X_2,X_0 \boxtimes X_1)}\<\bbD(v) \left( \bbD^2_{X_0,X_1}(\xi_0\odot \xi_1)\right),\xi_2\>_{\bbD(\bbOne)} \\
        &=\sum_{v \in O(X_2,X_0 \boxtimes X_1)}\bbD(R_{X_2})\left(\bbD^2_{\overline{X}_2,X_2} \left(j^\bbD_{X_2}(\left(\left(\bbD(v)(\bbD^2_{X_0,X_1}(\xi_0\odot\xi_1)\right) \right)) \odot \xi_2 \right)\right) .
    \end{align*}
        Let $(v^*)^\vee : = (\bar{R}_{X_2}^* \boxtimes \id_{\bar{X}_1 \boxtimes \bar{X}_0}) (\id_{\bar{X}_2} \boxtimes v^* \boxtimes \id_{\bar{X}_1 \boxtimes \bar{X}_0}) (\id_{\bar{X}_2} \boxtimes R_{X_0 \boxtimes X_1})$. From the axioms in the definition of C$^*$-algebra objects, it follows that
\[j^\bbD_{X_2}(\left(\left(\bbD(v)(\bbD^2_{X_0,X_1}(\xi_0\odot\xi_1)\right) \right)) = \left( \bbD((v^*)^\vee) \bbD^2_{\overline{X}_1,\overline{X}_0} (j^\bbD_{X_1}(\xi_1) \odot j^\bbD_{X_0}(\xi_0) )\right) . \]
        Consequently, 
        \begin{align*}
        \< \xi_0 \triangleright \xi_1, \xi_2 \>_{\bbD(\bbOne)} & = \sum_{v \in O(X_2,X_0 \boxtimes X_1)} \bbD(R_{X_2}) \left( \bbD^2_{\overline{X}_2,X_2} \left( \bbD((v^*)^\vee) \bbD^2_{\overline{X}_1,\overline{X}_0} (j^\bbD_{X_1}(\xi_1) \odot j^\bbD_{X_0}(\xi_0) )\right) \odot \xi_2 \right) \\
        & = \sum_{v \in O(X_2,X_0 \boxtimes X_1)} \bbD(R_{X_2}) (\bbD((v^*)^\vee )\odot \id)  \left( \bbD^2_{\overline{X}_1 \boxtimes \overline{X}_0,X_2} \left( \bbD^2_{\overline{X}_1,\overline{X}_0} ( j^\bbD_{X_1}(\xi_1) \odot j^\bbD_{X_0}(\xi_0))  \odot \xi_2  \right) \right) \\
        & = \sum_{v \in O(X_2,X_0 \boxtimes X_1)} \bbD(R_{X_2}) \left( \bbD((v^*)^\vee \boxtimes \id ) \left( \bbD^2_{\overline{X}_1, \overline{X}_0 \boxtimes X_2}\left( j^\bbD_{X_1}(\xi_1) \odot \bbD^2_{\overline{X}_0,X_2} (j^\bbD_{X_0}(\xi_0) \odot\xi_2)   \right) \right) \right) \\
         &= \sum_{v \in O(X_2,X_0 \boxtimes X_1)} \bbD(R_{X_2}) \left(\bbD(\id \boxtimes v^* ) \left( \bbD^2_{\overline{X}_1, \overline{X}_0 \boxtimes  X_2} \left( j^\bbD_{X_1}(\xi_1) \odot \bbD^2_{\overline{X}_0,X_2} (j^\bbD_{X_0}(\xi_0) \odot \xi_2)   \right) \right) \right) \\
        & \overset{(*)}{=} \sum_{w \in O(X_1,\overline{X}_0 \boxtimes X_2)} \bbD(R_{X_1}) \left( \bbD(\id \boxtimes w) \left( \bbD^2_{\overline{X}_1, \overline{X}_0 \boxtimes X_2} \left( j^\bbD_{X_1}(\xi_1) \odot \bbD^2_{\overline{X}_0,X_2} (j^\bbD_{X_0}(\xi_0) \odot \xi_2)   \right) \right) \right) \\
        & = \sum_{w \in O(X_1,\overline{X}_0 \boxtimes  X_2)} \bbD(R_{X_1}) \left( \left( \bbD^2_{\overline{X}_1, \overline{X}_0 \boxtimes X_2} \left( j^\bbD_{X_1}(\xi_1)\odot \bbD(w) \bbD^2_{\overline{X}_0,X_2} (j^\bbD_{X_0}(\xi_0) \odot \xi_2)   \right) \right) \right)  \\
        & = \< \xi_1, j^\bbD_{X_0}(\xi_0) \triangleright \xi_2 \>.
    \end{align*}
    In the fifth equality marked with $(*)$, we are making use of the canonical isomorphism 
    \[\cC(X_2, X_0  \boxtimes X_1) \overset{\text{Frob.}}{\simeq} \cC(\overline{X}_0 \boxtimes X_2,X_1) \overset{*}{\simeq} \cC(X_1, \overline{X}_0 \boxtimes  X_2) . \] 
\end{proof}

\begin{rem}
    As a consequence of the above computation, each map
    $$\xi_1 \mapsto \bbD(\omega) \left(\bbD^2_{X_0,X_1}(\xi_0 \odot \xi_1 )\right) , $$
    with $\omega \in O(X_2, X_0 \boxtimes X_1)$, is adjointable. We shall make use of this observation later.
\end{rem}

Applying the Lemma above to the C$^*$-algebra objects $\bbA$ and $\bbB$ simultaneously, we obtain the following.

\begin{cor}
    $|\bbA \times \bbB|$ is canonically and faithfully represented on $\cL_{\bbA(\bbOne) \otimes \bbB(\bbOne)}(\cE)$. Moreover, the vector $\Omega := 1_{\bbA(\bbOne)} \otimes 1_{\bbB(\bbOne)}$ is a cyclic vector for $| \bbA \times \bbB|$.
\end{cor}

\begin{defn}
    Define $\bbA \bowtie \bbB$ to be the norm closure of $|\bbA \times \bbB|$ in $\cL_{\bbA(\bbOne) \otimes \bbB(\bbOne)}(\cE)$. 
\end{defn}

Using the notation $\langle \cdot, \cdot \rangle$ for the $\bbA(\bbOne) \otimes \bbB(\bbOne)$-valued inner product on $\cE$, the formula
\[ \bbA \bowtie \bbB \ni T \mapsto \bbE(T) := \langle T \Omega, \Omega \rangle \]
defines a conditional expectation $\bbA \bowtie \bbB \to \bbA(\bbOne) \otimes \bbB(\bbOne)$. Our goal is to show that this conditional expectation is faithful.

\begin{lem}
    Let $\sum_{i=1}^n a_i \odot b_i \in \bbA(X) \odot \bbB(X)$. Then
    \[ \sum_{i,j=1}^n \bbA^2_{\overline{X},X}(j^\bbA_X(a_i) \odot a_j) \otimes \bbB^2_{\overline{X},X} (j^\bbB_X(b_i) \odot b_j) \]
    is positive in $\bbA(\overline{X} \boxtimes X) \otimes \bbB(\overline{X} \boxtimes X)$.
    \label{lemma:positivity1}
\end{lem}
\begin{proof}
    Consider on $\bbA(X)$ the canonical right Hilbert $\bbA(\overline{X} \boxtimes X)$-C*-module structure from Lemma \ref{lem:canonicalD(barXX)-modulestructure}. Observe that 
    \[ \sum_{i,j=1}^n \bbA^2_{\overline{X},X} (j^\bbA_X(a_i) \odot a_j) \otimes \bbB^2_{\overline{X},X} (j^\bbB_X(b_i) \odot b_j) \]
    is the $\bbA(\overline{X} \boxtimes X) \otimes \bbB(\overline{X} \boxtimes X)$-valued inner product of $\sum_i a_i \otimes b_i$ with itself in the exterior tensor product $\bbA(X) \otimes \bbB(X)$.
\end{proof}

\begin{lem}
    For $T \in \bbA(X) \odot \bbB(X) \subset \bbA \bowtie \bbB$, it holds
    \[ \| T \Omega \| \leq \| T \|_{\bbA \bowtie \bbB} \leq d_X^2 \| T \Omega \| \].
    \label{lem:Pimsner-Popaineq}
\end{lem}
\begin{proof}
    Writing $T = \sum_i a_i \odot b_i \in \bbA(X) \odot \bbB(X)$, we have 
    \[ \bbE(T^*T) = (E^\bbA_X \otimes E^\bbB_X) \left( \sum_{i,j} \bbA^2_{\overline{X},X}(j^\bbA_X(a_i) \odot a_j) \otimes \bbB^2_{\overline{X},X} (j^\bbB_X(b_i) \odot b_j) \right)  . \]
    Since both $E^\bbA_X$ and $E^\bbB_X$ are conditional expectations with Pimsner-Popa index at most $d_X^2$,  $E^\bbA_X \otimes E^\bbB_X$ has Pimsner-Popa index at most $d_X^4$. Moreover, the argument of $(E^\bbA_X \otimes E^\bbB_X)$ in the right-hand-side of the above inequality is positive by Lemma \ref{lemma:positivity1}. Thus,
    \[ \|T \Omega \|^2 = \|\bbE(T^*T) \| \leq \|T^*T\|_{\bbA \bowtie \bbB} \leq d_X^4 \| \bbE(T^*T) \| = d_X^4 \|T \Omega \|^2 . \]
\end{proof}

Now we prove Theorem \ref{thmalpha:condexponthecoendrealization}, which generalizes Proposition 3.7 in \cite{2023arXiv230505072H} to arbitrary C$^*$-algebra objects instead of just connected ones. 
In that case, the proof relied on the fibers being finite-dimensional, while for a general C*-algebra objects, we use the fact that its fibers are finite-index over the base algebra.

\begin{thm}[Theorem \ref{thmalpha:condexponthecoendrealization}]\label{thm:condexponthecoendrealization}
    The conditional expectation $\bbE: \bbA \bowtie \bbB \to \bbA(\bbOne) \otimes \bbB(\bbOne)$ is faithful.
    \label{thm:canonicalcondexpect}
\end{thm}
\begin{proof}
    Let $T \in \bbA \bowtie \bbB$ be such that $\bbE(T^*T) = 0$. Let $(T_\lambda)_{\lambda \in \Lambda}$ be a net in $|\bbA \times \bbB |$ converging to $T$ in norm. Write
    $$T_\lambda = \sum_{X \in \Irr(\cC)} T_\lambda^{(X)} ,$$
    with $T_\lambda^{(X)} \in \bbA(X) \odot \bbB(X)$. Observe that $E(T^*T) = 0$ if and only if $T \Omega = 0$. Let $P_X: \cE \to \bbA(X) \otimes \bbB(X)$ be the canonical projection. Then we have
    $$ T_\lambda^{(X)} \Omega = P_X T_\lambda \Omega  \to P_X T \Omega = 0 .$$
    From Lemma \ref{lem:Pimsner-Popaineq},  $T_\lambda^{(X)} \Omega \to 0$ implies $T_\lambda^{(X)} \to 0$ in the operator norm.
    Let $Y,Z \in \Irr(\cC)$. For $S \in \bbA(Y) \odot \bbB(Y)$, $T_\lambda^{(X)} S \Omega$ is supported on the direct summands $\bbA(W) \otimes \bbB(W)$ of $\cE$ for which $W \leq X \boxtimes Y$, meaning $\cC(W,X \boxtimes Y) \neq 0$. Using Frobenius reciprocity in $\cC$, we have then
    \[P_Z T_\lambda S \Omega = \sum_{X \leq Z \boxtimes \overline{Y}} P_Z T_\lambda^{(X)} S \Omega .\]
    Observe that the sum above is over a finite index set. It follows that
    $$P_Z T S \Omega = \lim_\lambda P_Z T_\lambda \Omega = \sum_{X \leq Z \boxtimes \overline{Y}}  \lim_\lambda P_Z T_\lambda^{(X)} S\Omega = 0 .$$
    Since 
    $$T S \Omega = \sum_{Z \in \Irr(\cC)} P_Z TS \Omega ,$$
    it follows that $TS \Omega = 0$. From the density of $\{S \Omega \ | S \in \bbA(Y) \odot \bbD(Y) , \ Y \in \Irr(\cC)\}$ in $\cE$, we conclude that $T = 0$.
\end{proof}

\begin{defn}
    An action of a unitary tensor category on a unital C$^*$-algebra $A$ is a unitary tensor functor $\alpha: \cC \to \fgpBim(A)$. 
\end{defn}

\begin{rem}
    In the above definition, $\alpha: \cC \to \fgpBim(A)$ is said to be an {\em outer action} if it is fully faithful. Such outer actions are of major interest in the literature, but our results can be proven in greater generality.
\end{rem}

Recall that an action $\cC \curvearrowright A$ of a unitary tensor category $\cC$ on a C$^*$-algebra $A$ can in particular be seen as a C$^*$-algebra object in $\cC^{\op}$.

\begin{defn}
\label{def:crossedproduct}
Let $\cC$ be a unitary tensor category and $A$ a unital C$^*$-algebra, and let $\alpha: \cC \to \fgpBim(A)$ be an action. Let $\bbD$ be a C$^*$-algebra object in $\Vec(\cC)$. The crossed product of $A$ by the pair $(\alpha, \bbD)$ is defined to be the coend realization 
\begin{align*}
    A \rtimes \bbD : = \alpha \bowtie \bbD.
\end{align*}   
\end{defn}

By Theorem \ref{thm:canonicalcondexpect}, there is a canonical faithful conditional expectation $\bbE: A \underset{\alpha}{\rtimes} \bbD \to A \otimes \bbD(\bbOne)$.

\subsection{Conditional expectations down to \texorpdfstring{$A$}{A}}

We say that a conditional expectation $E:A \rtimes \bbD \to A$ is compatible with the canonical conditional expectation $\bbE: A \rtimes \bbD \to A \otimes \bbD(\bbOne)$ if it is the composite of $\bbE$ with a conditional expectation $E: A \otimes \bbD(\bbOne) \to A$. The following Lemma will allow us to characterize such conditional expectations by means of states on the ground C$^*$-algebra $\bbD(\bbOne)$.

\begin{lem}
    Let $A$ and $D$ be unital C*-algebras and assume $A$ has trivial center. Then there is a one-to-one correspondence between ucp $A \- A$-bimodular maps $E:A \otimes D \to A$ and states on $D$, where $A$ acts on $A \otimes D$ via the embedding $A \mapsto A \otimes 1_D \subset A \otimes D$.
    \label{lem:correspondencebetweencondexpandstates}
\end{lem}
\begin{proof}
    Given a state $\omega$ on $D$, the corresponding ucp map is $E_\omega := \id_A \otimes \omega$. Now, if $E:A \otimes D \to A$ is a ucp $A\-A$-bimodular map, for any $d \in D$, the map
    \[ A \ni a \mapsto E_d(a) = E(a \otimes d) \]
    is a bounded $A\-A$-bimodular map $A \to A$. Since $\End_{A\-A}(A) \simeq \cZ(A) \simeq \bbC$, there is $\omega_E: D \to \bbC$ such that 
    \[E_d = \omega_E(d) \id_{A}. \]
    From the fact that $E$ is ucp one deduces that $\omega_E$ is a state.  Indeed, for all $d \in D$,
    \[\omega_E(d^*d)1_A = E_{d^*d}(1_A) = E(1_A \otimes d^*d) \geq 0 ,\]
    and $\omega_E(1_D)1_A = E(1_A \otimes 1_D) = 1$. 
    It is immediate that the constructions $\omega \mapsto E_\omega$ and $E \mapsto \omega_E$ are mutually inverses.
\end{proof}

\begin{prop}\label{prop:ExpectationsVSStates}
    Let $\alpha:\cC \to  \fgpBim(A)$ be an action of a unitary tensor category on a unital C*-algebra $A$ with trivial center. Let $\bbD$ be a C*-algebra object in $\Vec(\cC)$. Then there is an order preserving one-to-one correspondence between conditional expectations $A \rtimes \bbD \to A$ compatible with the canonical conditional expectation $A \rtimes \bbD \to A \otimes \bbD(\bbOne)$ and states on the C*-algebra $\bbD(\bbOne)$.
\end{prop}
\begin{proof}
     By Lemma \ref{lem:correspondencebetweencondexpandstates}, such conditional expectations are in bijections with $\cS(\bbD(\bbOne))$. It is easy to see that if $\omega_1 \leq \omega_2$ in $\cS(\bbD(\bbOne))$, then $E_{\omega_1} \leq E_{\omega_2}$.
\end{proof}

Observe that, in the Proposition above, the conditional expectation $E_\omega: A \rtimes \bbD \to A$ is faithful if and only if the corresponding state $\omega$ on $\bbD(\bbOne)$ is faithful.

\section{Flavors of discreteness}
\label{sec:flavorsofdiscreteness}

\subsection{Discrete, PQR and Ind inclusions}
\label{subsec:discretepqrin}

Let $A$ be a unital  C*-algebra with trivial center. Suppose we are given a unital inclusion $A \overset{E}{\subset}D$ and suppose $E: D \to A$ is a faithful conditional expectation. Let $\cE$ be the right $A$-$A$-correspondence obtained as the completion of $D$ in the $A$-valued inner product determined by $E$. We will use the notation $\cE = \overline{D}^{\| \cdot \|_A}$. The multiplication in $D$ extends to an action of $D$ on $\cE$, denoted by $(d,\zeta)\mapsto d \triangleright \zeta$ for $d \in D$ and $\zeta \in \cE$. Letting $\Omega$ be the unit of $D$ seen as a vector in $\cE$, the canonical embedding $D \to \cE$ is given by $d \mapsto d \triangleright \Omega$.   

Let $K \in \fgpBim(A)$. There is a distinguished class of adjointable $A\-A$-bilinear maps $K \to \cE$, those lying in the diamond space
\[\rCorr_{A\-A}(K,\cE)^{\diamondsuit}:= \{ f \in \rCorr_{A\-A}(K,\cE) \ | \ f[K] \subset D \} . \] 
Given $f \in \rCorr_A(K,\cE)^{\diamondsuit}$, there is a unique $A$-$A$-bimodular map $\check{f}: K \to D$ such that $f(\xi) = \check{f}(\xi) \triangleright \Omega$ for all $\xi \in K$.
The next Theorem, named C$^*$-Frobenius reciprocity, shows that the diamond space functor $\rCorr_{A\-A}((-),D)^\diamondsuit: \fgpBim(A)^{\op} \to \Vec$ has a canonical structure of a C$^*$-algebra object. 

\begin{thm}[\cite{2023arXiv230505072H}, Theorem 2.9]    \label{thm:Frobeniusrecip}
    Let $A \overset{E}{\subset} D$ be as above. Let $\cE = \overline{D}^{\| \cdot\|_A}$ and let $\Omega$ be the image of $1_D$ in $\cE$. For $K \in \fgpBim(A)$, the maps
    \[ \Psi_K: \rCorr_{A\-A}(K,D)^\diamondsuit \leftrightarrow \rCorr_{A \- D} (K \underset{A}{\boxtimes} D,D): \Phi_K , \]
    given by
    \[ \Psi_K(f)( \xi \boxtimes d) := \check{f}(\xi)d \ \ \text{and} \ \  \Phi_K(g)(\xi) := g( \xi \boxtimes 1_D) \triangleright \Omega , \]
    are mutually inverses. They are moreover natural in $K$, and therefore are mutually inverses in $\Vec(\fgpBim(A))$:
    \[\Psi: \rCorr_{A\-A}((\-),D)^\diamondsuit \leftrightarrow \rCorr_{A \- D} ((\-) \underset{A}{\boxtimes} D,D): \Phi\]
\end{thm}

The isomorphism $\Psi: \rCorr_{A\-A}((\-),D)^\diamondsuit \simeq \rCorr_{A \- D} ((\-) \underset{A}{\boxtimes} D,D)$ provides a Banach space structure on the diamond space. From the assumption $\cZ(A) \simeq \bbC$, it follows, on the other hand, that $\rCorr_{A\-A}((\-),D)$ has a canonical Hilbert space structure. Generally, the Banach space structure of $\rCorr_{A\-A}((-),\cE)^\diamondsuit$ may not agree with the Hilbert space structure of $\rCorr_{A\-A}((-),\cE)$. This is close in spirit to the distinction between operator norm and $\ell^2$-norm in the GNS-construction for C$^*$-algebras.

\begin{cor}[\cite{2023arXiv230505072H}, Theorem 2.7]
The functor $\rCorr_{A\-A}((\-),D)^\diamondsuit: \cC^{\op} \to \Vec$ has a canonical structure of a C$^*$-algebra object.
\end{cor}

C$^*$-discreteness will be characterized by the fact that the images of the maps in $\rCorr_{A}(K,\cE)^{\diamondsuit}$ form a dense subspace of $D$ as $K$ runs in $\Irr(A)$. Then C*-Frobenius reciprocity gives a interpretation of the diamond spaces, and hence of discrete inclusions, in terms of categorical data.

Suppose we are given a Hilbert space object $\cH$ in the unitary ind-completion $\Hilb(\fgpBim(A))$ of $\fgpBim(A)$. Consider, for each $K \in \Irr(A)$, the amplified $A\-A$-correspondence $K \otimes \cH(K)$.

\begin{defn}    \label{def:realizationofHilbertspaceobjects}
    The realization $|\cH|$ of $\cH \in \Hilb(\fgpBim(A))$ is the $A\-A$-correspondence
    \[ |\cH| := \bigoplus_{K \in \Irr(A)}^{\ell^2} K \otimes \cH(K) . \]
\end{defn}

\begin{prop}
    Realization of Hilbert space objects extends to a fully-faithful unitary tensor functor $|-|: \Hilb(\fgpBim(A)) \to \rCorr_{A\-A}$, realizing $\Hilb(\fgpBim(A))$ as a full subcategory of $\rCorr_{A\-A}$.
    \label{prop:realizationofHilbertspaceobjectsfunctor}
\end{prop}
\begin{proof}
    If $f: \cH_1 \to \cH_2$ is a morphism of Hilbert space objects, the morphism $|f|: |\cH_1| \to |\cH_2|$ given on the $K$-th component by $\id_K \otimes f_K$ defines the action of $|-|$ on morphisms. We have thus a functor $\Hilb(\fgpBim(A)) \to \rCorr_{A\-A}$.

    Let us show that this functor is a unitary tensor functor. Given Hilbert space objects $\cH_1$ and $\cH_2$, we have
    \begin{align*}
        |\cH_1 \boxtimes \cH_2| &= \bigoplus_{K \in \Irr(A)}^{\ell^2} K \otimes (\cH_1 \boxtimes \cH_2)(K) \\
        &= \bigoplus_{K \in \Irr(A)}^{\ell^2} K \otimes \left( \bigoplus_{K_1,K_2 \in \Irr(A)}^{\ell^2} \rCorr_{A\-A}(K,K_1 \underset{A}{\boxtimes} K_2) \otimes \cH_1(K_1) \otimes \cH_2(K_2) \right) \\
        & \simeq \bigoplus_{K,K_1,K_2 \in \Irr(A)}^{\ell^2} \rCorr_{A\-A}(K,K_1 \underset{A}{\boxtimes} K_2) \otimes K \otimes \cH_1(K_1) \otimes \cH_2(K_2) \\
        &\simeq \bigoplus_{K_1,K_2 \in \Irr(A)}^{\ell^2} K_1 \underset{A}{\boxtimes} K_2 \otimes \cH_1(K_1) \otimes \cH_2(K_2) \simeq |\cH_1| \underset{A}{\boxtimes} |\cH_2| ,
    \end{align*}
    and all isomorphisms are canonical and unitary, providing $|-|$ with a unitary tensor structure, associativity being straightforward to check.

    It remains to show that $|-|$ is fully-faithful. We compute
    \begin{align*}
        \rCorr_{A\-A}(|\cH_1|, |\cH_2|) &= \rCorr_{A\-A}\left(\bigoplus_{K_1 \in \Irr(A)}^{\ell^2} K_1 \otimes \cH_1(K_1),\bigoplus_{K_2 \in \Irr(A)}^{\ell^2} K_2 \otimes \cH_2(K_2)\right) \\
        & \simeq \prod_{K \in \Irr(A)}^{\ell^\infty}\rCorr_{A\-A}\left( K \otimes \cH_1(K), K \otimes \cH_2(K) \right) \\
        & \simeq \prod_{K \in \Irr(A)}^{\ell^\infty} \bbB(\cH_1(K),\cH_2(K)) = \Hilb(\fgpBim(A))(\cH_1,\cH_2) . 
    \end{align*}
\end{proof}

From the data $A \overset{E}{\subset} D$, we can extract three  ind-objects of the unitary tensor category $\fgpBim(A)$.

\begin{defn}
    The Hilbert space object $\cH_\cE$ associated to $A \overset{E}{\subset} D$ is given by
    \[ \cH_\cE(K) := \rCorr_{A\-A}(K, \cE) . \]

    The C*-algebra object $\bbD$ associated to $A \overset{E}{\subset} D$ with faithful expectation $E$ is given by
    \[ \bbD(K) := \rCorr_{A\-A} (K, \cE)^{\diamondsuit} . \]
\end{defn}

Canonically, we have $\bbD(K) \subset \cH_\cE(K)$ for all $K \in \fgpBim(A)$.

The third ind-object we consider is the fiberwise closure of $\bbD$ inside $\cH_\cE$. It can also be characterized as a GNS-construction by means of the Lemma below.

\begin{lem}
    Let $A \overset{E}{\subset} D$ and $\bbD$ be as above. Then the conditional expectation $E$ induces a canonical faithful state on the C*-algebra object $\bbD$.
\end{lem}

\begin{proof}
    We have the $*$-algebra inclusions
    \[ A \subset A \otimes \bbD(\bbOne) \subset \bigoplus_{K \in \Irr(A)} K \odot \bbD(K) \subset D. \]
    The restriction of the conditional expectation $E: D \to A$ to a conditional expectation $A \otimes \bbD(\bbOne) \to A$ is given by $(\id_A \otimes \omega)$ for a unique state $\omega$ on $\bbD(\bbOne)$ (Lemma \ref{lem:correspondencebetweencondexpandstates}).
\end{proof}

\begin{defn}[\cite{JP17}, Section 4.5.]
The Hilbert space object $L^2_\omega \bbD$ is the GNS-construction  associated to the state $\omega = \omega_E$ on $\bbD$ induced by the conditional expectation $E: D \to A$.
\label{defn:GNSHilbertspace}
\end{defn}

The fibers of $L^2_{\omega}\bbD$ are described as follows. Recall that, for every $K \in \fgpBim(A)$, $\bbD(K)$ is a $\bbD(\bbOne)\-\bbD(\bbOne)$-correspondence. Composing the right $\bbD(\bbOne)$-valued inner product with $\omega$ endows $\bbD(K)$ with a pre-Hilbert space structure, and $L^2_{\omega}\bbD(K)$ is the completion of this pre-Hilbert space. By applying the realization in Definition \ref{def:realizationofHilbertspaceobjects}, we obtain an $A$-$A$-correspondence $|L^2_\omega \bbD|$.

\begin{defn}    \label{def:notionsofregularity}
    Let $A$ and $D$ be unital C*-algebras. Assume in addition that $A$ has trivial center. A faithful inclusion $A \overset{E}{\subset} D$ is said to be
    \begin{itemize}
        \item [(1)] C$^*$-Discrete if $A \rtimes \bbD \simeq D$; 
        
        \item [(2)] Projective Quasi-Regular (PQR) if $\cE \simeq |L^2_\omega \bbD|$;
        \item [(3)] an Ind-Inclusion if $\cE \simeq |\cH_\cE|$.
    \end{itemize}
\end{defn}

For each $K \in \Irr(A)$, by C$^*$-Frobenius reciprocity \ref{thm:Frobeniusrecip}, 
\[\rCorr_{A\-A}(K,\cE)^\diamondsuit \simeq \rCorr_{A\-D}(K \underset{A}{\boxtimes} D, D), \]
naturally in the variable $K$. Using this natural identification, we shall regard $\rCorr_{A\-A}(K,\cE)^\diamondsuit$ as a Banach space. On the other hand, $\cH_\cE(K) = \rCorr_{A\-A}(K,\cE)$ is a Hilbert space, the inner-product given by $\< f, g \>_{\cH_\cE(K)} = f^* g$, and $\rCorr_{A\-A}(K,\cE)^\diamondsuit \subset \rCorr_{A\-A}(K,\cE)$ as a linear subspace.

\begin{lem}
    The closure of $\bbD(K) = \rCorr_{A\-A}(K,\cE)^\diamondsuit$ in the Hilbert space $\cH_\cE(K)$ coincides with $L^2_\omega\bbD(K)$.
\end{lem}
\begin{proof}
    Given $f \in \bbD(K)$, its $\cH_\cE(K)$-norm $\| f\|_{\cH_\cE(K)}$ is defined by
    \[  f^*f(\xi) =  \| f\|_{\cH_\cE(K)}^2 \xi \qquad  \forall \ \xi \in K, \]
    or equivalently $\< f(\xi),f(\xi)\> = \| f\|_{\cH_\cE(K)}^2 \| \xi \|^2$ for all $\xi \in K$. But
    \begin{align*}
        \< f(\xi),f(\xi)\> & = \< \check{f}(\xi) \Omega, \check{f}(\xi) \Omega \> \\
        & = \< (\check{f}(\xi) )^* \check{f}(\xi) \Omega, \Omega \>
        \\
        & = E ( (\check{f}(\xi) )^* \check{f}(\xi)) \\
        & = (\id \otimes \omega_E)\bbE^\bbD (\check{f}(\xi) )^* \check{f}(\xi)). 
    \end{align*}
    In the identification of $A \rtimes \bbD$ with a C$^*$-subalgebra of $D$, the element $\check{f}(\xi) )^* \check{f}(\xi)$ corresponds to 
    \[ \left( \bar{\xi} \odot j^\bbD(\check{f}) \right) \cdot \left(\xi \odot \check{f} \right), \]
    and therefore
    \[ (\id \otimes \omega_E)\bbE^\bbD (\check{f}(\xi) )^* \check{f}(\xi)) = \| \xi \|^2 \omega_E(\bbE^\bbD(j^\bbD(\check{f}) \cdot \check{f})) , \]
    where we are writing $\bbD^2(j^\bbD(\check{f}) \odot \check{f}) = j^\bbD(\check{f}) \cdot \check{f}$ for short. Since this must hold for every $\xi \in K$, it follows that
    \[ \| f\|_{\cH_\cE(K)}^2 = \omega_E(\bbE^\bbD(j^\bbD(\check{f}) \cdot \check{f})) . \]
\end{proof}

\begin{prop}
In Definition \ref{def:notionsofregularity}, $(1)$ implies $(2),$ and $(2)$ implies $(3)$. 
\end{prop}
\begin{proof}
    The closure of $A \rtimes \bbD \subset D$ inside $\cE$ coincides with the image of $|L^2_\omega(\bbD)|$ in $\cE$. Since $D$ embeds densely in $\cE$, $A \rtimes \bbD \simeq D$ implies $|L^2_\omega(\bbD)| \simeq \cE$. This shows $(1) \implies (2)$. It is obvious that $(2) \implies (3)$, since $L^2_\omega \bbD$ is a Hilbert space object.
\end{proof}

\begin{rem}
    If $A \overset{E}{\subset} D$ is an irreducible inclusion, i.e., $A' \cap D \simeq \bbC$, then Projective Quasi Regularity and C$^*$-discreteness agree with definitions 2.3 and 3.8 given in \cite{2023arXiv230505072H}.
\end{rem}

\subsection{Characterization of C*-discrete inclusions}
\label{subsec:reconstruction}

In the spirit of categorical duality for compact quantum group actions \cite{MR3204665} and of subfactor reconstruction, in this Section we characterize C$^*$-discrete inclusions by means of categorical data.

\begin{defn}
Given a unitary tensor category $\cC$, let $\rCorr_\Omega(\Vec(\cC))$ be the category whose objects are pairs $(\bbD,\omega)$, where $\bbD$ is a C*-algebra object and $\omega$ is a faithful state on $\bbD$ and whose morphisms are state-preserving ucp maps $\theta: (\bbD_1,\omega_1) \to (\bbD_2,\omega_2)$ is a morphism iff $\theta$ is a ucp map $\bbD_1 \to \bbD_2$ and $\omega_2 \circ \theta = \omega_1$.
\end{defn}
\begin{defn}
Given a unital C$^*$-algebra $A$ with trivial center, let $\CDisc(A)$ denote the category whose objects are discrete inclusions $A \overset{E}{\subset} D$. Morphisms between objects $(A \subset D_1, E_1)$ and $(A \subset D_2, E_2)$ are $A$-$A$-bimodular ucp-maps $\phi: D_1 \to D_2$ mapping the conditional expectation $E_1$ to $E_2$, i.e., $E_2 \circ \phi = E_1$.
\end{defn}

\begin{rem}
In Section \ref{sec:examples}, the symbol $\Omega$ will also be used to denote cyclic vectors for correspondences. Below, a state on a C$^*$-algebra object will allow for a {\em realization} of the abstract categorical data as concrete operator algebraic data. This is analogous to the GNS-construction, justifying the overuse of the symbol $\Omega$.
\end{rem}

The following theorem is an extension of \cite[Theorem 4.3]{2023arXiv230505072H} to non-connected C*-algebra objects.

\begin{thm}[Theorem \ref{thmalpha: thereconstructiontheorem}]\label{thm: thereconstructiontheorem}
    There is an equivalence
    \[ \rCorr_\Omega(\Vec(\fgpBim(A))) \simeq \CDisc(A). \]
\end{thm}

\begin{proof}
        At the level of objects, the functor $\rCorr_\Omega(\Vec(\fgpBim(A))) \to \CDisc(A)$ is given by taking the crossed-product.  Functoriality is proven as in Theorem 5.7 in \cite{2023arXiv230407155H}, but we give a sketch of the proof for completeness.

        Suppose $\theta: (\bbD_1, \omega_1) \to (\bbD_2,\omega_2)$ is a morphism in $\rCorr_\Omega(\Vec(\fgpBim(A)))$. Define 
        \begin{align*}
            |\theta|:& |\bbD_1| \to |\bbD_2|\\
            & \xi \odot \eta \mapsto \xi \odot \theta(\eta).
        \end{align*}
        It extends to a ucp-map $\id_A \rtimes \theta: A \rtimes \bbD_1 \to A \rtimes \bbD_2$. Let $\pi_2: \bbD_2 \to \bbB(L^2_{\omega_2}\bbD_2)$ be the GNS-representation. The composite $\pi_2 \circ \theta$ is a ucp map $\bbD_1 \to \bbB(L^2_{\omega_2}\bbD_2)$. For $\xi \odot \eta \in K \odot \bbD_1(K)$, write
        \[ |\pi_2 \circ \theta|(\xi \odot \eta) := \xi \odot \pi_2 (\theta(\eta)). \]
        Then $|\pi_2 \circ \theta|$ extends to a ucp-map $A \rtimes \bbD_1 \to \cL_{A}(|L^2_{\omega_2} \bbD_2|)$ in such a way that
        \begin{center}
            \begin{tikzcd}
A \rtimes \bbD_1 \arrow[rr, "\id_A \rtimes \theta"] \arrow[rdd, "| \pi_2 \circ \theta|"'] &                                    & A \rtimes \bbD_2 \arrow[ldd, "\text{can.}"] \\
                                                                                          &                                    &                                             \\
                                                                                          & \cL_{A}(|L^2_{\omega_2}\bbD_2|) &                                            
\end{tikzcd}
        \end{center}
commutes; in the diagram, $\text{can.}$ denotes the canonical representation of $A \rtimes  \bbD_2$ on $| L^2_{\omega_2} \bbD_2|$. Denoting by $E_i: A \rtimes \bbD_i \to A$ the conditional expectations corresponding to $\omega_i$, since $\omega_2 \circ \theta_1 = \omega_1$, it follows
\[ E_2 \circ (\id_A \rtimes \theta) = E_1. \]
We remark that, in contrast with \cite{2023arXiv230407155H} and \cite{MR3948170}, the above proof of functoriality is shorter because complete positivity implies norm continuity. In the cited papers, the context was a W$^*$-context, and normality of the realization of the maps had to be shown.

Let us construct the functor $\CDisc(A) \to \rCorr_\Omega(\Vec(\fgpBim(A)))$. At the level of objects, it is given by taking the underlying C*-algebra object associated to the discrete inclusion, together with the faithful state on it provided by the faithful conditional expectation (Proposition \ref{prop:ExpectationsVSStates}). We are left to check functoriality with respect to morphisms. Suppose $\phi: (A \subset D_1, E_1) \to (A \subset D_2,E_2)$ is a morphism of discrete inclusions. Using that $\phi$ sends $E_1$ to $E_2$ and Kadisons' inequality, we obtain
    \[ E_2 (\phi(x^*) \phi(x)) \leq E_2(\phi(x^*x)) = E_1(x^*x) \qquad \forall \ x \in D_1. \]
As a consequence, $\phi$ extends to a bounded $A$-$A$-bimodular map $\Phi: \cE_1 \to \cE_2$, where $\cE_i := \overline{D_i}^{\| \cdot \|_{A_i}}$. Let $\bbD_i$ be the underlying C*-algebra object of the discrete inclusion $(A \subset D_i, E_i)$, and let $\omega_i$ be the faithful state on $\bbD_i$ provided by the conditional expectation $E_i$. Recall that
    \[ \bbD_i(K) = \rCorr_{A\-A}(K,\cE_i)^\diamondsuit, \]
    and, since $A \subset D_i$ is assumed C$^*$-discrete, which implies PQR,
    \[L^2_{\omega_i}\bbD_i(K) = \rCorr_{A\-A}(K,\cE_i),  \]
where $K \in \fgpBim(A)$. For every such $K$, and for every $f \in \rCorr_{A}(K,\cE_1)$, $\Phi \circ f$ is a bounded $A\-A$-bimodular map, and therefore adjointable since $K$ is fgp. We have thus, for every $K \in \fgpBim(A)$, a well defined linear map
    \[ \rCorr_{A\-A}(K,\cE_1)  \to \rCorr_{A\-A}(K,\cE_2) \]
given by $f \mapsto \Phi \circ f$.
Since $\Phi$ extends $\phi: D_1 \to D_2$, it preserves the diamond spaces, and defines therefore a natural transformation $\bbD_1 \to \bbD_2$. Since $\phi$ is an $A\-A$-bimodular ucp map, the resulting morphism $\bbD_1 \to \bbD_2$ is easily seen to be a ucp map of C*-algebra objects.

It is straightforward to check that the functors $\rCorr_\Omega(\Vec(\fgpBim(A))) \leftrightarrow \CDisc(A)$ are mutually quasi-inverses.
\end{proof}

\subsection{Characterization of ind-inclusions}
\label{subsec:indinclusions}

Let $A \overset{E}{\subset} D$ be an ind-inclusion of unital C*-algebras, where $A$ is assumed to have trivial center. Our goal in this section is to characterize the algebra
\[ \rCorr_{A\-A}(\cE) \]
where $\cE$ is as in Section 4, i.e., the right $A\-A$ correspondence obtained from $D$ and the conditional expectation $E: D \to A$. We will show in this section that a decomposition of $\rCorr_{A\-A}(\cE)$ into a direct product of type I factors characterizes $A \overset{E}{\subset}D$ as being an ind-inclusion. 

Let $\cH_\cE$ be the object in $\Hilb(\fgpBim(A))$ associated to $\cE$. Recall $A \overset{E}{\subset} D$ being an ind-inclusion means
\[ \cE \simeq |\cH_\cE| := \bigoplus_{K \in \Irr(A)}^{\ell^2} K \otimes \cH_\cE(K) \ \, \]
as a right $A\-A$-correspondence. In particular, $K \otimes \cH_\cE(K)$ is a complemented $A\-A$-subcorrespondence of $\cE$ for all $K \in \Irr(A)$. Let us denote by $P_K$ the corresponding orthogonal central projection. The net of projections which associates to every finite subset $\Lambda \subset \Irr(A)$ the projection
    \[ \sum_{K \in \Lambda} P_K \]
converges strictly to the identity $\id_\cE$.

\begin{lem}    \label{lemma:relcommutcorners}
    For each $K \in \Irr(A)$, 
    \[ P_K \rCorr_{A\-A}(\cE) \simeq \cL(\cH_\cE(K)). \]
\end{lem}

\begin{proof}
    It follows from the above orthogonal decomposition of $\cE$ that 
    \[P_K \rCorr_{A\-A}(\cE) = P_K \rCorr_{A\-A}(\cE)P_K \simeq \rCorr_{A\-A}(K \otimes \cH_\cE(K)). \]
    Since $K$ is a simple $A\-A$-bimodule, it follows that
    the latter is $\cL(\cH_\cE(K))$.
\end{proof}

\begin{cor}\label{cor:relcommutcharact->}
    There is an isomorphism
    \[  \rCorr_{A\-A}(\cE) \simeq \prod_{K \in \Irr(A)}^{\ell^\infty} \cL(\cH_\cE(K)). \]
\end{cor}
\begin{proof}
    It follows from the discussion above that there is an embedding of $\rCorr_{A\-A}(\cE)$ into $\prod_K^{\ell^\infty} \cL(\cH_\cE(K))$.
    Now an element $x \in \prod_K^{\ell^\infty} \cL(\cH_\cE(K))$ is a bounded family $(x_K)_{K \in \Irr(A)}$ where $x_K \in \cL(\cH_\cE(K))$. There is a unique $\tilde{x} \in \rCorr_{A\-A}(\cE)$ defined by the condition that
    \[ P_k \tilde{x} = (\id_K \otimes x_K) P_K \qquad \forall \ K \in \Irr(A). \]
    Then the image of $\tilde{x}$ under the embedding $\rCorr_{A\-A}(\cE) \hookrightarrow \prod_K^{\ell^\infty} \cL(\cH_\cE(K))$ coincides with $x$.
\end{proof}

We now show the results above characterize  ind-inclusions.

\begin{thm}[Theorem \ref{thmalpha:relcommutchar<-}]\label{thm:relcommutchar<-}
    The inclusion $A \overset{E}{\subset}D$ is an ind-inclusion if and only if there exists a family of mutually orthogonal central projections $\{P_i\}_{i \in I} \subset \rCorr_{A\-A}(\cE)$ such that
    \begin{itemize}
        \item[(1)] $\sum_{i \in I} P_i$ converges strictly to $\id_\cE$;
        \item[(2)] for each $i \in I$ there is a Hilbert space $H_i$ such that $P_i \rCorr_{A\-A}(\cE) \simeq \cL(H_i)$;
        \item[(3)] for every projection $P \in \rCorr_{A\-A}(\cE)$, $P(\cE) \in \Hilb(\fgpBim(A))$.
    \end{itemize}
\end{thm}
\begin{proof}
    We already know that $A \subset D$ being and ind-inclusion implies the existence of $\{P_i\}_{i \in I}$ with the properties above. Let us prove the converse.

    Given $i \in I$, $(3)$ implies that
    \[ P_i( \cE) = \bigoplus_{K \in \Irr(A)}^{\ell^2} K \otimes H_i(K), \]
    where $H_i \in \Hilb(\fgpBim(A))$. Since for $i \neq j$ in $I$ we have $P_i P_j = P_j P_i = 0$, the supports of $H_i$ and $H_j$ in $\Irr(A)$ must be disjoint. ($K \in \Irr(A)$ belongs to the support of $H_i$ means that $H_i(K) \neq 0$.) Thus, if $K$ maps non-trivially into $\cE$ as an $A\-A$-bimodule, there is a unique $i \in I$ such that $K$ belongs to the support of $H_i$. For $i,j \in I$,
    \[ \Hilb(\fgpBim(A))(H_i,H_j)  = \prod_{K \in \Irr(A)}^{\ell^\infty} \cL(H_i(K),H_j(K)). \]
    Now the orthogonality of the family $\{P_i\}_i$ together with item $(2)$ imply that for each $i$ there is a unique $K_i$ in the support of $H_i$. Moreover, there is an embedding
    \[ \bigoplus_{i \in I} K_i \otimes H_i(K_i) \hookrightarrow \cE. \]
    Item $(1)$ implies that this embedding has dense range, that is,
    \[ \cE \simeq \bigoplus_{i \in I}^{\ell^2} K_i \otimes H_i(K_i). \]
    Thus, defining the Hilbert space object $\cH_\cE$ by $\cH_\cE(K_i) = H_i(K_i)$, $\cH_\cE(K) = 0$ if $K$ is not in the support of any of the $H_i$'s, it follows
    \[ \cE \simeq |\cH_\cE|. \]
\end{proof}

\section{Examples}
\label{sec:examples}

\subsection{Compact quantum group actions}

Now we explain how the theory of continuous actions of compact quantum groups fits into our framework, providing examples of C$^*$-discrete inclusions. For the fundamentals of the theory of compact quantum groups and their representation categories, see \cite{MR3204665}.

Let $\bbG$ be a compact quantum group with reduced algebra of functions $C(\bbG)$, i.e., we consider the Haar state to be faithful on $C(\bbG)$. Consider the unitary tensor category $\Rep(\bbG)$ of finite dimensional unitary representations of $\bbG$, together with its fiber functor $F: \Rep(\bbG) \to \Hilb_{fd}$, which is in particular a C$^*$-algebra object in $\Vec(\Rep(\bbG)^{\op})$. Then the coend realization
\[ F \bowtie (-): \bbA \mapsto F \rtimes \bbA =: A \]
can be equipped with the structure of a functor from the category of C$^*$-algebra objects in $\Vec(\Rep(\bbG))$ to the category of C$^*$-algebras equipped with a {\em continuous} action of $\bbG$ (\cite{MR3426224}). There is a canonical isomorphism
\[ \bbA(\bbOne) \simeq A^{\bbG}, \]
where $A^{\bbG}$ denotes the fixed point subalgebra of $A$ under the $\bbG$-action, and $\bbOne \in \Rep(\bbG)$ stands for the trivial representation. The canonical conditional expectation
\[ \bbE: A = F \bowtie \bbA \to F(\bbOne) \otimes \bbA(\bbOne) \simeq \bbA(\bbOne) \simeq A^{\bbG} \]
corresponds to averaging along the $\bbG$-action on $A$ by means of the Haar state.

\subsection{Semi-circular systems}
\label{subs:semicircularsystems}

In this section we describe when an $A$-valued semi-circular system gives rise to an ind-inclusion. This is equivalent to asking when the Fock space $\cF(\eta)$ of the given covariance matrix is an ind-object in $\fgpBim(A)$. Our main result here is a characterization of that property in terms of the covariance matrix $\eta$.

\begin{defn}
    A covariance matrix on a C*-algebra $A$ is a completely positive map $\eta: A \to A \otimes \cL(\ell^2(I))$, $I$ being an arbitrary index set.
\end{defn}

Covariance matrices can be produced by means of the following Lemma.

\begin{lem}
    Suppose $\cE$ is a right $A\-A$-correspondence. Suppose $\{\xi_i\}_{i \in I}$ is a family of vectors in $\cE$ with the property that there exists $C \geq 0$ such that
    \[ \sum_{j \in I} \| \< \xi_i, a \triangleright \xi_j \> \|^2 \leq C \|a\|^2 \ \forall \ a \in A ,  \ \forall \  i \in I . \]
    Denoting by $\{e_{i,j}\}_{i,j \in I}$ a choice of a system of matrix units in $\cL(\ell^2(I))$, it follows that 
    \begin{align*}
        \eta: A & \to A \otimes \cL(\ell^2(I)) \\
        a & \mapsto \sum_{i,j \in I} \< \xi_i, a \triangleright \xi_j \> \otimes e_{i,j}
    \end{align*}
    is a covariance matrix on $A$.
\end{lem}

\begin{proof}
    In the canonical representation $A \otimes \cL(\ell^2(I)) \subset \cL_{A}(A \otimes \ell^2(I))$, it is straightforward to see that  
    \[\sum_{j \in I} \| \< \xi_i, a \triangleright \xi_j \> \|^2 \leq C \|a\|^2 \]
    for all $i \in I$ implies $\|\eta(a)\| \leq C^{1/2} \|a\|$, and that $\eta(a^*a)$ is positive for all $a$. To show that $\eta$ is completely positive, observe that $\eta \otimes \id_{\bbM_n}$ can be written in terms of the correspondence $\cE \otimes \bbC^n$ as
    \[(\eta \otimes \id_{\bbM_n})( a \otimes m_{kl}) = \sum_{i,j \in I} \sum_{k',l' = 1}^n \< \xi_i \otimes v_{k'}, (a \otimes m_{k,l}) \triangleright (\xi_j \otimes v_{l'} \> \otimes m_{k',l'} \otimes e_{i,j} , \]
    where $\{v_k\}_k$ is an orthonormal basis for $\bbC^n$ and $\{m_{k,l}\}_{k,l}$ the corresponding system of matrix units.
\end{proof}

 The discussion to follow shows in particular that every covariance matrix can be obtained in this way. 
 
 Associated to a covariance matrix $\eta$, there is an $A\-A \otimes \cL(\ell^2(I))$-correspondence 
\[A \underset{\eta}{\otimes} (A \otimes \cL(\ell^2(I))) . \]
It is the separation-completion of the algebraic tensor product $A \odot (A \otimes \cL(\ell^2(I)))$ with respect to the inner product
\[ \langle a \otimes x , b \otimes y \rangle := x^* \eta(a^*b) y . \]

Fix $i_0 \in I$. There is an associated embedding of C*-algebras
\[ A \simeq A \otimes \bbC e_{i_0,i_0} \subset A \otimes \cL(\ell^2(I)) , \]
It is clear then that
\[ \cX := [A \underset{\eta}{\otimes} (A \otimes \cL(\ell^2(I)))] \triangleleft (1 \otimes e_{i_0,i_0}) \]
inherits the structure of a right $A\-A$-correspondence. Observe that, if we write
\[ \eta(a) = \sum_{i,j} \eta_{i,j}(a) \otimes e_{i,j} , \]
and define $\xi_i := 1 \otimes e_{i,i_0} \in \cX$, then
\[ \langle \xi_i, a \triangleright \xi_j \rangle = \eta_{i,j}(a) . \]

\begin{defn}
    The Fock space associated to the covariance matrix $\eta$ above is
    \[ \cF(\eta) := \bigoplus_{n \in \bbZ_{\geq 0}}^{\ell^2}\cX^{\underset{A}{\boxtimes} n}  , \]
    where $\cX^{\underset{A}{\boxtimes}0 } : = A$.
\end{defn}

Given $\xi \in \cX$, let $T_\xi \in \cL_A(\cF(\eta))$ be the creation operator associated to $\xi$:
\[ T_{\xi}(\eta) = \xi \boxtimes \eta \ \forall \ \eta \in \cF(\eta) . \]

\begin{defn}
    The semi-circular operator $X_\xi \in \cL_A(\cF(\eta))$ associated to a vector $\xi \in \cX$ is given by 
    \[ X_\xi := T_\xi + T_\xi^* . \]
    For $i \in I$, write $X_i := X_{\xi_i}$. Let
    \[ \Phi(\eta) := C^*(A,\{X_i\}_{i \in I}) \subset \cL_A(\cF(\eta)) . \]
\end{defn}

The image $\Omega$ of the unit $1_A \in A$ under the embedding $A \to \cX \to \cF(\eta)$ is a cyclic vector for the action of $\Phi(\eta)$. It induces a conditional expectation $E: \Phi(\eta) \to A$ via
\[ E(x) := \langle x \Omega, \Omega \rangle . \]

From \cite[Theorem 5.2]{2024arXiv240918161H}, a sufficient condition for the conditional expectation $E$ to be faithful is that $A$ admits a faithful tracial state $\tau$ such that
$\tau(\eta_{ij}(x)y) = \tau(x\eta_{ji}(y))$ for all $x,y \in A$. 

Assume from now on that $E$ is faithful. Let $\bbD_\eta: \fgpBim(A)^{\op} \to \Vec$ be given by
\[ \bbD_\eta(K) := \rCorr_{A\-A}(K,\cF(\eta))^{\diamondsuit} := \{ f \in \rCorr_{A\-A}(K, \cF(\eta)) \ | \ f[K] \subset \Phi(\eta) \Omega\} . \]

Due to C*-Frobenius reciprocity, $\bbD_\eta \in \rCorr(\Vec(\fgpBim(A)))$. There is a canonical embedding $A \rtimes \bbD_\eta \hookrightarrow \Phi(\eta)$, given by
\begin{align*}
    K \otimes \bbD_\eta(K) & \to \Phi(\eta) \\
    \zeta \otimes f &\mapsto \check{f}(\zeta) ,
\end{align*}
where, for $K$ a simple $A\-A$-bimodule, $\check{f}$ denotes the unique factorization of $f$ through $\Phi(\eta) \ni x \mapsto x \triangleright \Omega \in \cF(\eta)$.

  \begin{lem}
  We have
  \[ \overline{\spann\{f(\xi) \ | f \in \bbD_\eta(K), \xi \in K, K \in \Irr(A) \}}^{\|\cdot\|_A} \simeq |L^2_\omega \bbD_\eta | , \]
 as $A$-$A$ correspondences, where $\omega$ is the state associated to the conditional expectation $A \rtimes \bbD \to A$ obtained by restriction of the conditional expectation on $\Phi(\eta)$.
  \end{lem}
\begin{proof}
Recall the $*$-algebra
\[ | A \times \bbD_\eta | := \bigoplus_{K \in \Irr(A)} K \odot \bbD_\eta(K) , \]
and the faithful conditional expectation $\bbE: | A \times \bbD_\eta | \to A \odot \bbD_\eta(\bbOne) \subset A \otimes \bbD_\eta(\bbOne)$. The state $\omega \in \bbD_\eta(\bbOne)$ is determined by $E = (\id_A \otimes \omega) \circ \bbE$.
    Let $x_\lambda$ be a net in $|A \times \bbD|$ such that $x_\lambda \triangleright \Omega \to 0$ in $\cF(\eta)$. This is the case if and only if
    \[ E(x_\lambda^* x_\lambda) = (\id \otimes \omega)(\bbE(x_\lambda^* x_\lambda)) \to 0 . \]
    The right hand side is exactly the $A$-valued inner product on $|L^2_\omega \bbD_\eta |$.
\end{proof}

\begin{thm}
    Suppose the conditional expectation $E: \Phi(\eta) \to A$ is faithful. Then the inclusion $A \subset \Phi(\eta)$ is an ind-inclusion if and only if $A \underset{\eta_{ii}}{\otimes} A \in \Hilb(\fgpBim(A))$. 
    \label{thm:characindsemicirc}
\end{thm}
\begin{proof}
    From \ref{prop:realizationofHilbertspaceobjectsfunctor}, we deduce that $A \underset{\eta_{ii}}{\otimes} A \in \Hilb(\fgpBim(A))$ if and only if
    \[ A \underset{\eta_{ii}}{\otimes}A \simeq |\rCorr_{A\-A}((-),A \underset{\eta_{ii}}{\otimes}A )| . \]
    We also have that $A \underset{\eta_{ii}}{\otimes}A \simeq \overline{A \triangleright \xi_i \triangleleft A}^{\|\cdot \|_A} = \overline{A \triangleright (X_i \triangleright \Omega )\triangleleft A}^{\|\cdot \|_A} $. Thus, A $\underset{\eta_{ii}}{\otimes} A \in \Hilb(\fgpBim(A))$ if and only if the $A\-A$-correspondence generated by $X_i \triangleright \Omega$ is a Hilbert space object. A simple computation show that
    \[\overline{A \triangleright X_j X_i \triangleright \Omega \triangleleft A}^{\|\cdot \|_A} \subset (A \underset{\eta_{jj}}{\otimes} A) \underset{A}{\boxtimes}(A \underset{\eta_{ii}}{\otimes}A) \oplus A .\]
    Since realization is a unitary tensor functor (Proposition \ref{prop:realizationofHilbertspaceobjectsfunctor}), the right hand side of the above inclusion is a Hilbert space object, and therefore the left hand side also is. By iteration, we conclude that the $A\-A$-correspondence generated by $P\triangleright \Omega$, where $P$ is a polynomial in the semi-circular elements, is a Hilbert space object if and only if the $A \underset{\eta_{ii}}{\otimes} A$'s are Hilbert space objects. Passing to the limit, we have that $\overline{\Phi(\eta)\triangleright \Omega}^{\| \cdot \|_A}$ is a Hilbert space object if and only if the $A\underset{\eta_{ii}}{\otimes} A$'s are.
\end{proof}

\begin{cor}\label{cor:characterizationindsemicircularintermsofcovariancematrix}
    Suppose the conditional expectation $E: \Phi(\eta) \to A$ is faithful. Then the $A$-valued semi-circular system $A \subset \Phi(\eta)$ is a ind-inclusion if and only if $\eta$ can be written
    \[ \eta_{ij}(a) = \langle \xi_i, a \triangleright \xi_j \rangle \]
    for a family $\{\xi_i\}_{i \in I}$ of vectors in an ind-object $H \in \Hilb(\fgpBim(A))$, satisfying
    \[ \sum_{j \in I} \| \< \xi_i, a \triangleright \xi_j \> \|^2 \leq C \|a\|^2 \ \forall \ a \in A , \ \forall \  i \in I, \]
    for some $C \geq 0$.
\end{cor}

\begin{ex}
    Let $\{\alpha_i\}_{i\in I}$ be a countable collection of automorphisms on a unital C*-algebra $A$ with faithful unique trace $\tau$, and consider the covariance matrix
    $$
    \eta=\mathsf{diag}(\eta_i),
    $$
    where $\eta_i= \alpha_i + \alpha_i^{-1}$. 
    Then, by direct computation we see that for every $i \in I$ and all $x,y\in A$ 
    $$\tau(x\eta_i(y)) = \tau(\eta_i(x)y),$$
    fulfilling the hypotheses of \cite[Theorem 5.2]{2024arXiv240918161H}. 
    Then, the inclusion $A\subset \Phi(\eta)$ has a faithful conditional expectation and is C*-discrete. 
\end{ex}

\subsection{C*-algebraic factorization homology}
\label{subs:facthom}

Factorization homology, as developed in \cite{MR3431668}, is a generalized homology theory for topological manifolds. It is a homology theory in the sense that it is functorial and satisfy an excision property, and is generalized in the sense that it assigns categories as invariants of manifolds. In \cite{2023arXiv230407155H}, it was proven that as a target for factorization homology one can take the symmetric monoidal category $\Clin$, a certain category of C$^*$-categories, equipped with the max tensor product $\underset{\max}{\boxtimes}$ as introduced in \cite{MR4162123}. Our goal here is not to give a detailed account on that subject, but briefly explain how it can provide examples of non-connected C$^*$-algebra objects. 

We will be interested in factorization homology for surfaces. A factorization homology theory with values in $(\Clin, \underset{\max}{\boxtimes})$ is a monoidal functor $\cF: \Surf \to \Clin$, where the source category is the category of surfaces with embeddings as morphisms and disjoint unions as monoidal structure. The functoriality of $\cF$ is expressed in 2-categorical terms: smooth oriented embeddings correspond to unitary functors between C*-tensor categories, and isotopies between embeddings correspond to unitary natural isomorphisms between unitary functors. This implies, in particular, that if $D$ denotes the compact oriented disk, then
\[ \cF(D) =: \cC \]
is a unitarily braided C*-tensor category. Also, $\cF$ must be functorial with respect to smooth oriented embeddings $\Sigma_1 \to \Sigma_2$, and due to the co-completeness of $\Clin$ (see \cite{MR4162123}), $\cF$ is completely determined by $\cC$. For this reason, we call $\cF$ the factorization homology with coefficients in $\cC$, and the notation
\[ \cF(\Sigma) =: \int_\Sigma \cC \]
is used.

The values of factorization homology are determined by universal properties, and can in principle be difficult to compute. But assuming that $\cC$ is a unitary tensor category equipped with a unitary braiding $\tau$, there are very concrete computations.

\begin{thm}[\cite{2023arXiv230407155H}, Theorem 4.5]
    Let $\Sigma$ be an oriented compact surface with non-empty boundary. Each choice of boundary component $B \subset \Sigma$ induces canonically a structure of cyclic right $\cC$-module C*-category on $\int_\Sigma \cC$. Thus, there is a C*-algebra object $\bbD_{\Sigma} \in \Vec(\cC)$ such that
    \[ \int_\Sigma \cC \simeq \bbD_{\Sigma} \- \Mod_\cC , \]
    as cyclic right $\cC$-module C*-categories.
    \label{thm:C*algebrasfromfacthom}
\end{thm}

\begin{ex}
    If $\Sigma = Ann$ is the annulus, then
    \[ \bbD_{Ann} = \bigoplus_{X \in \Irr(\cC)}^{c_0} \overline{X} \boxtimes X , \]
    i.e., as a functor $\cC^{\op} \to \Vec$, 
    \[ \bbD_{Ann}(Y) = \bigoplus_{X \in \Irr(\cC)}^{c_0} \cC(Y, \overline{X} \boxtimes X) . \]
    The algebra structure of $\bbD_{Ann}$ can be described as follows. It is determined by the morphisms
    \[ \overline{X} \boxtimes X \boxtimes \overline{Y} \boxtimes Y \overset{\tau_{\overline{X}\boxtimes X,\overline{Y}} \boxtimes \id_Y}{\longrightarrow} (\overline{Y} \boxtimes \overline{X}) \boxtimes (X \boxtimes Y) \longrightarrow \overline{W} \boxtimes W , \] 
    for $X,Y,W \in \Irr(\cC)$, where the last arrow is given by irreducible decomposition. Observe that the vector space dimension of $\bbD_{Ann}(\bbOne)$ is the cardinality of $\Irr(\cC)$, so that $\bbD_{Ann}$ is in general not connected nor locally finite. More generally, the C$^*$-algebra objects $\bbD_{\Sigma}$, for $\partial \Sigma \neq \emptyset$, can be described as sort of twisted tensor products of copies of $\bbD_{Ann}$.
\end{ex}

For $\Sigma$, as above, with a fixed non empty boundary component $B \subset \partial \Sigma$, consider the associated mapping class group $\Gamma(\Sigma,B)$. It is the group of oriented self-diffeomorphisms of $\Sigma$ fixing $B$ pointwise, modulo isotopies.

Given now an action $\cC \to \fgpBim(A)$ on a unital C*-algebra $A$ that has trivial center, we have the associated crossed products $A \rtimes \bbD_{\Sigma}$. 

\begin{thm}
    There is a canonical action of $\Gamma(\Sigma,B)$ on $A \rtimes \bbD_\Sigma$ by $*$-automorphisms such that
    \[ A \subset (A \rtimes \bbD_\Sigma)^{\Gamma(\Sigma,B)} . \]
    \label{thm:crossedproductswithmappingclassgroupactions}
\end{thm}

Theorem \ref{thm:crossedproductswithmappingclassgroupactions} follows from \cite[Proposition 5.13]{2023arXiv230407155H}, and functoriality of the crossed products. It is analogous to \cite[Corollary 5.14]{2023arXiv230407155H}.

The conditional expectations $A \rtimes \bbD_\Sigma \to A$ of interest would be the $\Gamma(\Sigma,B)$-invariant ones. Those correspond to $\Gamma(\Sigma,B)$-invariant states on $\bbD_\Sigma(\bbOne)$. We do not know if they exist in general, but we have an example; for the annulus, the corresponding mapping class group is then isomorphic to $\bbZ$.

\begin{thm}[\cite{2023arXiv230407155H}, Theorem 5.15]
    The projection
    \[ \bbD_{Ann} \simeq \bigoplus_{X \in \Irr(\cC)}^{c_0} \overline{X} \boxtimes X \to \bbOne \]
    induces a $\bbZ$-equivariant state on $\bbD_{Ann}(\bbOne)$. 
\end{thm}

\section{von Neumann vs C*-discreteness}
The discrete property of a unital inclusion of operator algebras $(C\subset D,E)$ have been formulated in terms of a quasi-regularity condition for both subfactors and inclusions of C*-algebras. Here, $E:D\to C$ is a conditional expectation; ie a (normal) ucp map onto $C$. Namely, an inclusion is quasi-regular if certain intermediate $*$-algebra whose structure is inherited from a unitary tensor category (UTC) is dense in the appropriate sense. In this section, we develop the framework for which a direct comparison of these two notions---von Neumann-discreteness and C*-discreteness---imply one another. See \cite{2023arXiv230505072H} and \cite{MR3948170} for complete details about discrete inclusions. 

Let $\cC$ be a UTC, $A$ be a unital C*-algebra, and $N$ be a  $\rm{II}_1$-factor with its unique trace $\tr.$ Recall that an action of $\cC$ on $(N, \tr)$ (or on $A$) is given by a unitary tensor functor $F'':N\to \spbfBim(N)$ (or $F:\cC\to\fgpBim(A)$). 
Here, $\spbfBim(N)$ are the finitely generated projective (\emph{aka} bifinite) $N$-$N$ bimodules which are \emph{spherical}; ie whose left and right von Neumann dimensions match. (See \cite[Example 2.10]{MR4139893} for a detailed description of $\spbfBim(N)$.) We say that {\bf these actions are compatible} if and only if the following conditions hold: 
\begin{enumerate}[label=($\cH$\arabic*)]
    \item There exists a fully faithful unitary tensor functor $$\cH:\fgpBim(A)\to \spbfBim(N)$$ such that 
    $$\cH\circ F\overset{\eta}{\cong} F'',$$
    where $\eta$ is some monoidal natural unitary isomorphism. (Later in this manuscript, we will often suppress $\eta$ and assume it is the identity.) Here, the monoidality is with respect to $\{(\cH F)^2_{X,Y}:= \cH(F^2_{X,Y})\circ \cH^2_{FX,FY}\}_{X,Y\in\cC}$. Namely, 
    $$
    \eta_{X\otimes Y}\circ (\cH F)^2_{X,Y} = (F''_{X,Y})^2\circ (\eta_X\boxtimes \eta_Y).
    $$

    \item\label{Hilb:Dense} There is a collection of contractive $A$-$A$ bimodular injective maps   
    $$
    \{\Omega^X:F(X)\hookrightarrow \left((\cH\circ F)(X)\right)^\circ\}_{X\in\cC}
    $$
    with $\Omega^X[F(X)]\subset (\cH\circ F)(X)$ dense for every $X\in \cC.$ 
    Here, $\left((\cH\circ F)(X)\right)^\circ$ denotes the dense subspace of left (equivalently right) $N$ \emph{bounded vectors} $\xi\in (\cH\circ F)(X)$, which are defined by the property that the map $N\ni n\mapsto  n\rhd\xi\in (\cH\circ F)(X)$ extends to a bounded map on $L^2(N,\tr)$ \cite{Bisch97} . 

    \item\label{Hilb:Tensor} The family $\{\eta_Z\circ \Omega^Z:F(Z)\to F''(Z)^\circ\}_{Z\in\cC}$ of contractive $A$-$A$ bimodular maps is compatible with the actions $F$ and $F''$. This is, for all $X,Y\in \cC$ the following identity holds: 
    $$
        (F''_{X,Y})^2\circ(\eta_X\boxtimes \eta_Y)\circ (\Omega^X\boxtimes\Omega^Y) = \eta_{X\otimes Y}\circ\Omega^{X\otimes Y}\circ F^2_{X,Y},
    $$
    the equality holding at the level of bounded vectors.

    \item\label{Hilb:F1} We assume $(\eta_1\circ\Omega^1)[F(\bbOne_\cC)]\subset N\Omega = (F''(1))^\circ$ unitally, and that in this representation, $A''\subset \End(L^2(N,\tr))$ is a factor. This is, $\Omega^1$ is a map such that $(\eta_1\circ \Omega^1)(1_A)= \Omega$ is the cyclic and separating vector of the left action of $N$ on $F''(1)= L^2(N,\tr).$ 
    We then identify $A\subset N\subset \End(L^2(N,\tau))$ unitally, and get a ${\rm II}_1$-subfactor $A''\subseteq N$. 
\end{enumerate}
\noindent The functor $\cH$ is meant to encode the properties of a GNS construction that is coherent with the dynamical and topological data. 

\begin{ex}\label{ex:GJSHilbertification}
An important class of examples arises from the Guionnet-Jones-Shlyakhtenko (GJS) construction. 
Using these ideas, in \cite{BHP12} it was shown that given any unitary tensor category $\cC$ there is a $\rm{II}_1$-factor $N$ (which can be taken to be $L\bbF_\infty$) and an outer action $F'':\cC\to\fgpBim(N).$ 
Moreover, in \cite{MR4139893} the second-named author and Hartglass established a similar result for C*-algebras, constructing a monotracial unital simple C*-algebra $(A,\tr)$ from any UTC $\cC$, together with an outer action $F:\cC\to \fgpBimtr(A).$ 
The superscript $\tr$ here denotes that the action of $\cC$ on $A$ is by {\bf bimodules which are compatible with the trace}; this is for each $c\in\cC$ and each $\xi,\eta\in F(c)$ the identity $$\tr\left(\langle \eta,
\xi\rangle_A\right)= \tr\left({_A}\langle \xi, \eta\rangle\right)$$ holds. 
This is a technical assumption used to extend the left and right $A$-actions to normal left and right $N$-actions. 
Moreover, they also constructed a \emph{Hilbertification functor} $$(-\boxtimes_N L^2(N)):\fgpBimtr(A)\to \spbfBim(N),$$
which is a fully faithful unitary tensor functor taking a completion in $2$-norm with respect to the trace, witnessing that $F$ and $F''$ are manifestly compatible actions. 
\end{ex}

\begin{lem}\label{lem:A''isN}
    Identifying $A$ with its image in $L^2(N,\tr)$ under $\Omega_1$ we have that 
    $$A''=N\subset\End(L^2(N,\tr)).$$
\end{lem}
\begin{proof}
By \cite[Lemma 1.5.11]{MR2391387}, there is a unique trace-preserving normal conditional expectation $\tilde{E}:N\to A''$ onto the von Neumann finite subfactor $A''.$ 
Moreover, $\tilde E$ is implemented by the Jones projection $f\in \End(L^2(N,\tr))$ whose range is $L^2(A,\tr).$ 
However, by assumption \ref{Hilb:F1}, we have  $L^2(A,\tr)=L^2(N,\tr),$ and so $f=\id_{L^2(N)}.$ 
Thus, by \cite[Lemma 3.1]{MR829381}, $A''=N.$ 
\end{proof}

\begin{remark}
    In practice, it will often be the case that we encounter $A''=N$ directly, and so the factoriality assumption of \ref{Hilb:F1} holds automatically  (c.f. Example \ref{ex:GJSHilbertification}). 
    In Lemma \ref{lem:A''isN}, assuming \ref{Hilb:F1} we relied on $A''$ being a factor to conclude $A''=N$ using Kosaki's result.
    However, one may as well start assuming that $A\subset N$ is strongly dense. 
    We well however not further explore those starting  technical assumptions here. 
\end{remark}

In the context of compatible actions we may speak of extending the action of $\cC\overset{F}{\curvearrowright}A$ to an action $\cC\overset{F''}{\curvearrowright}N$. 
Notice that by Condition \ref{Hilb:F1}, the unique trace $\tr$ on $N$ restricts to a faithful trace $\tr|A$ on $A.$ 
Similarly, one can view this as restricting an action $\cC\overset{F''}{\curvearrowright}N$ to an action on a C*-subalgebra $\cC\overset{F}{\curvearrowright}A$. 
Later, in Example \ref{ex:GJSHilbertification} we will describe a family of outer UTC-actions on unital monotracial C*-algebras which can be extended/restricted compatibly into outer UTC-actions on a $\rm{II}_1$-factor.

By \cite[Theorem 4.3]{2023arXiv230505072H}, connected C*-discrete inclusions $(A\subset B, E)$ correspond to connected C*-algebra objects in  $\cC_{A\subset B},$ the support UTC. 
Similarly by \cite[Theorem 5.35]{MR3948170}, connected $N$-discrete inclusions $(N\subset M, E)$ correspond to connected W*-algebra objects in $\cC_{N\subset M}.$ However, connected W*-algebra objects are given by a family of finite-dimensional W*-algebras graded by a UTC, and so connected W*-algebra objects are the same as connected C*-algebra objects \cite[Definition 2.4]{MR3948170}, and therefore it suffices to consider the latter. 

Let $\bbB$ be a connected C*-algebra object in $\Vec(\cC)$, and fix compatible outer actions (ie fully faithful) $F$ and $F''$ as above.
Using the realization functor from \cite{MR3948170} we obtain an $N$-discrete inclusion (i.e. a discrete subfactor)
$$\left(N\subset N\rtimes_{F''}\bbB =:M,\ E^M_N\right).$$  
We briefly outline this construction: Taking bounded vectors $F''(c)^\circ$ for each $c\in\cC$ yields a $*$-algebra object ${\bbF''}^\circ: \cC\to \Vec$. With it, on the one hand, we form the pre-Hilbert space 
$$|\bbB|_{\bbF''}^\circ:=\bigoplus_{Z\in\Irr(\cC)} {F''(Z)}^\circ \otimes_\bbC \bbB(Z),$$
and complete it with respect to the inner product induced by $(\tr\otimes\tau)\circ \mathsf{Proj}_{1_\cC}$---the composite of $\tr\otimes \tau$ with the projection onto the $1_\cC$-graded component---, obtaining the Hilbert space $L^2|\bbB|_{\bbF''}.$ 
Here, $\tau$ identifies $\bbB(1_\cC)\cong \bbC$ by connectedness, and so $\tr\otimes\tau$ is faithful. 
On the other hand, using the $*$-algebra object structures of $\bbF''$ and $\bbB$ we endow the vector space $|\bbB|_{\bbF''}^\circ$ with the structure of a $*$-algebra, denoted $|{\bbF''\times \bbB}|$, which acts on $L^2|\bbB|_{\bbF''}$  by bounded operators by left multiplication, denoted $\rhd.$ 
We define $M$ as the double commutant of $(|{\bbF''\times \bbB}|\rhd)\subset \End(L^2|\bbB|_{\bbF''}),$ and $E^M_N$ is the normal conditional expectation down to $N$, extending $\mathsf{Proj}_{\bbOne_\cC}.$
Details of this construction can be found on \cite[\S 5]{MR3948170}. 

In the following, we will use Sweedler's notation $\xi_{(Z)}\otimes f_{(Z)}$ to denote an arbitrary vector in the $Z$-graded component $F''(Z)^\circ \otimes \bbB(Z),$ where the reader should keep present this vector might not be elementary, but rather a sum thereof. 

Similarly, using the reduced crossed product construction from \cite{2023arXiv230505072H}, extended here in Chapter \ref{sec:generalrealization}, we obtain an $A$-discrete inclusion $$\left(A\subset  A\rtimes_{r,F}\bbB =:B,\ E^B_A\right).$$ 
Here, $E^B_A$ is a faithful conditional expectation onto $A$, which picks up only the $1_\cC$-graded component of a vector at the algebraic level. 
The composition of $E^B_A$ with the restricted trace on $A$ gives a faithful state on $B$ 
$$\varphi:=\tr\circ E^B_A.$$ 
Therefore, $B$ acts faithfully on the GNS space $L^2(B,\varphi)$ by bounded operators extending the left multiplication, with cyclic and separating vector $\Omega^\varphi$ given by the image of $1_A=1_B$ under the GNS map. 
Considering $B'' \subset \End(L^2(B,\varphi))$ we obtain a second  inclusion of von Neumann algebras $$A''\subset B'',$$ which is the smallest inclusion of von Neumann algebras extending $A\subset B$.

We now show that $A''\subset B''$ is a subfactor equipped with a normal expectation. 
\begin{lem}\label{lem:ExtensionE}
    Under the representation $B\subset\End(L^2(B,\varphi))$, we have that 
    \begin{enumerate}
        \item As von Neumann factors $$A''= N.$$
        \item There exists $ {E^B_A}'':B''\to A''$ unital faithful normal conditional expectation onto $A''.$
        \item The state $\varphi=\tr\circ E^M_N$ admits a normal extension to $B'',$ still denoted $\varphi.$ Furthermore, ${E^B_A}''$ is compatible with $\varphi.$
    \end{enumerate}
\end{lem}
\noindent It is therefore meaningful to determine whether or not this inclusion is of discrete type and isomorphic to $\left(N\subset M,\ E^M_N\right)$. 
\begin{proof}
    (1): Let us denote by $(A\rhd)\subset \End(L^2(N))$, so that $(A\rhd)''=N$ as in Lemma \ref{lem:A''isN}. Also, we denote by $(A\blacktriangleright)\subset \End(L^2(B,\varphi)).$ We shall show that $(A\blacktriangleright)''\cong N.$ 
    First, we observe that for all $a_1, a_2, a \in A$ the identity
    $$\langle a\rhd a_1\Omega,  a_2\Omega\rangle = \tr(a_1^*a^* a_2) = \tr\circ E^B_A(a_1^*a^* a_2) = \langle a\blacktriangleright a_1\Omega^\varphi,  a_2\Omega^\varphi\rangle$$ 
    holds. 
    By \cite[Lemma 3.40]{MR4139893} (alternatively use \cite[Proposition 4.15]{MR3948170}) applied to the $*$-algebra $A$, the Hilbert space $L^2(B,\varphi)$, and the Hilbert space $L^2(N,\tr)$ with dense subset $A\Omega$, we conclude that the left action $(A\blacktriangleright)$ extends to a normal left action  $([(A\rhd)'']\blacktriangleright)\subset \End(L^2(B,\varphi)).$ 
    
    Since $[A\blacktriangleright]'' \supset [A\blacktriangleright]\subset [N\blacktriangleright]$, it follows at once that $[A\blacktriangleright]''\subseteq [N\blacktriangleright],$ since the latter is a von Neumann algebra containing $[A\blacktriangleright].$  
    We shall now show that this is an equality. 
    For each arbitrary $n\in N$, there is a net $a_\lambda\in A$ such that $(a_\lambda\rhd) \to n\in \End(L^2(N,\tr))$, as $\lambda\to \infty$, in the weak operator topology. 
    By the above extension, we have that for any $a\in A,$ we have $(a\rhd)\blacktriangleright = a\blacktriangleright,$ and so we obtain that in the strong operator topology 
    $$
        (n\blacktriangleright) = \left(\left( \lim_\lambda (a_\lambda\rhd)\right)\blacktriangleright\right) = \lim_\lambda ((a_\lambda\rhd)\blacktriangleright)  = \lim_\lambda (a_\lambda\blacktriangleright).
    $$
    In the second equality above, we used that the representation $\blacktriangleright$ is continuous with respect to the weak operator topology, and so commutes with $\mathsf{WOT}$-limits.
    Therefore, 
    $$\blacktriangleright:[A\rhd]''\to [A\blacktriangleright]''$$ defines a surjective normal unital $*$-morphism of von Neumann algebras. 
    Since $[A\rhd]''=N$ is a factor, then $\blacktriangleright$ is injective, and so $[A\rhd]''\cong[A\blacktriangleright]''$ as von Neumann factors.
\smallskip
           
    (2): Recall that $E^B_A$ is implemented by the Jones projection $e_A\in \End^\dag(\cB_A)$, with $e_A[\cB]= A\Omega.$ 
    Here, $\cB$ is the right $B$-$A$ correspondence obtained by completing $B$ in the $A$-norm $\|\cdot\|_A$ from the conditional expectation $E^B_A$  and $\Omega$ denotes the image of $1_A$ under this representation (i.e. a cyclic vector for the left $B$-action on $\cB$).
    Notice that since the 2-norm obtained from $\varphi= \tr\circ E^B_A$ satisfies for all $b\in B,$ $\|b\Omega^\varphi\|_2\leq \|b\Omega\|_A.$ 
    Thus, the $L^2$-extension of $e_A$ matches with the  usual Hilbert space Jones projection from $L^2(B,\varphi)$ onto the closed subspace $L^2(A,\tr)= L^2(N,\tr)$ denoted $e_N$. 
    And so, for $b''\in B''$ we define ${E^B_A}''(b'')$ as the unique operator in $\End(L^2(N,\tr)_N)$ defined by the following relation on $L^2(N,\tr):$
    $$
        \left({E^B_A}''(b'')\right)\circ e_N = e_N\circ b''\circ e_N
        .
    $$
    Indeed, we have a $*$-isomorphism
    \begin{align}\label{eqn:ImplementationE''}
        \End (L^2(N,\tr)_N)&\cong N\\
        f&\mapsto (f(\Omega))^\vee\nonumber\\
        L_n&\mapsfrom n,\nonumber
    \end{align}
    where $L:\underbrace{(L^2(N,\tr))^\circ}_{=N\Omega}\to\End(L^2(N,\tr)_N)$ maps $n\to L_n$, given on $N\Omega$ by $L_n(x\Omega)=nx\Omega.$ 
    Moreover, since $f$ preserves bounded vectors in $L^2(N)$, it follows that $f(\Omega)\in N\Omega,$ and so $(f(\Omega))^\vee\in N$ is uniquely specified by $f(\Omega) = (f(\Omega))^\vee\Omega.$ 
    The maps $n\mapsto L_n$ and $f\mapsto (f(\Omega)^\vee$ are readily seen to be each other's inverses. 
    Therefore, Equation (\ref{eqn:ImplementationE''}) yields a well-defined function ${E^B_A}'':B''\to A''$, which is readily seen to be unital, surjective, completely positive and normal, since it is implemented by conjugating by the Jones projection. 
    Furthermore, by normality, it follows that ${E^B_A}''(b'')\Omega^\varphi = e_A(b''\Omega^\varphi)$, for all $b''\in B''.$
    Thus, ${E^B_A}''$ yields the desired conditional expectation.

\smallskip

    (3) Follows, extending $\varphi$ by the expression $\tr \circ{E^B_A}'',$ which on the nose is a normal and faithful state on $B''.$ 
    This state is compatible with the extended conditional expectation by construction. 
\end{proof}

Before we state the next lemma, we notice that by construction, the Hilbert space $L^2|\bbB|_{\bbF''}$ has the structure of an $N$-$N$ bimodule. 
We will use the fact that $A''\cong N$, and that both $L^2|\bbB|_{\bbF''}$ and $L^2(B,\varphi)$ contain isometric, and dense by \ref{Hilb:Dense}, copies of $B^\diamondsuit:= \oplus_Z\bbF(Z)\otimes \bbB(Z)$ to extend the $A$-$A$ bimodule---by left/right multiplication--- structure on the latter Hilbert space to that of an $N$-$N$ bimodule with normal actions.  
For arbitrary $a\in A$ and $b\in B,$ we have that 
$$
\|b\Omega^\varphi \lhd a\|_2^2=\tr(a^*\underbrace{E^B_A(b^*b)}_{=c^*c}a) = \tr(caa^*c^*)\leq \|a\|^2\tr(c^*c)=\|a\|^2\cdot\|b\Omega^\varphi\|_2^2.
$$
Which implies the right action $(\lhd A)$ on $L^2(B, \varphi)$ is by bounded operators, and similarly for  the left $A$-action. 
A direct application of \cite[Lemma 3.40]{MR4139893} affords us the desired normal extension, thus making $L^2(B, \varphi)$ into an $N$-$N$ bimodule. 

\begin{lem}\label{lem:UnitaryNN}
    The map     \begin{align}\label{eqn:BOmegatoB''Omega}
        \bigoplus_{Z\in\Irr(\cC)}\bbF(Z)\otimes \bbB(Z)&\to \bigoplus_{Z\in\Irr(\cC)}\bbF''(Z)^\circ\otimes\bbB(Z)\\
        \xi_{(X)}\otimes f_{(X)}&\mapsto \Omega^X(\xi_{(X)})\otimes f_{(X)}\nonumber
    \end{align}
    extends to an $N$-$N$ bilinear unitary $U$ \begin{align}\label{eqn:UnitaryIntertwiner}
        U:L^2(B,\varphi)&\to L^2|\bbB|_{\bbF''}\\
        \left(b_{(X)}\otimes f_{(X)}\right)\Omega^{\varphi}&\mapsto \Omega^X (b_{(X)})\otimes f_{(X)}\nonumber
    \end{align}
    for all $b\in B.$
\end{lem}
\begin{proof}
    By \ref{Hilb:Dense}, the map from Expression (\ref{eqn:BOmegatoB''Omega}) is well-defined and its range is dense in $L^2|\bbB|_{\bbF''}$. 
    Furthermore, it is isometric with respect to the $\|\cdot\|_2$-norms on both sides. 
    Indeed, for $\sum_X b_{(X)}\otimes f_{(X)}$ in the algebraic sum, we have that 
    \begin{align*}
    \left\|\sum_X(b_{(X)}\otimes f_{(X)})\Omega^\varphi \right\|_{L^2(B,\varphi)} &= \sum_{X,Z}(\tr\otimes\tau)\left(E^B_A \left( j^{\bbF}_X(b_{(X)})\otimes j_{X}^{\bbB}(f_{(X)})\cdot (b_{(Z)}\otimes f_{(Z)}) \right) \right)\\ &= \sum_{X,Z}(\tr\otimes \tau)\left(E^M_N \left( j^{\bbF}_X(b_{(X)})\otimes j_{X}^{\bbB}(f_{(X)})\cdot (b_{(Z)}\otimes f_{(Z)}) \right) \right)\\ &= \left\| \sum_X b_{(X)}\Omega^X\otimes f_{(X)} \right\|_{L^2|\bbB|_{\bbF''}},  
    \end{align*}
    since $E^M_N$ and $E^B_A$ coincide on the algebraic sum. 
    (Recall that $\tau$ is the map identifying $\bbB(1_\cC)\cong \bbC$, since $\bbB$ is assumed connected.) 
    Thus, the map (\ref{eqn:BOmegatoB''Omega}) extends to an isometry with dense range $U$ as in Expression (\ref{eqn:UnitaryIntertwiner}). 
    Thus $U$ is a unitary. 

    We shall  now show that $U$ is right $N$-linear. 
    Let $n\in N$ and $(a_\lambda)_\lambda\subset A$ converging strongly to $n$ on $L^2(B,\varphi),$ as $\lambda\to \infty.$ 
    We then have that for any arbitrary $\sum_{Z\in \Irr(\cC)}\xi_{(Z)}\otimes \eta_{(Z)} \in L^2(B, \varphi)$: 
    \begin{align*}
        \left( \sum_{Z\in \Irr(\cC)}\xi_{(Z)}\otimes \eta_{(Z)}  \right)\blacktriangleleft n &= \lim_\lambda \left( \left( \sum_{Z\in \Irr(\cC)}\xi_{(Z)}\otimes \eta_{(Z)}\right)\blacktriangleleft a_\lambda  \right)\\ 
        &= \lim_\lambda \left(\sum_{Z\in \Irr(\cC)} \left(\xi_{(Z)}\blacktriangleleft a_\lambda\right)\otimes \eta_{(Z)} \right)\\
        &\overset{U}{\mapsto}\ \lim_\lambda U\left(\sum_{Z\in \Irr(\cC)} \left(\xi_{(Z)}\lhd a_\lambda\right)\otimes \eta_{(Z)} \right)\\
        &= \lim_\lambda \left(\left(U\left(\sum_{Z\in \Irr(\cC)} \xi_{(Z)}\otimes \eta_{(Z)}\right)\right) \lhd a_\lambda \right)\\
        &=  \left(U\left(\sum_{Z\in \Irr(\cC)} \left(\xi_{(Z)}\right)\otimes \eta_{(Z)} \right)\right) \lhd n.
    \end{align*}
    Here, we have used that $U$ is $\|\cdot\|_2$-continuous and commutes with the right $A$-action on the nose, which is also continuous. 
    The last equality follows by the normality of the right $N$-action on $L^2|\bbB|_{\bbF''}.$
    Thus, $U$ is right $N$-linear

    The left $N$-linearity is similar. 
\end{proof}

\begin{lem}
    The unitary $U$ is left $B^\diamondsuit$-linear.
\end{lem}
\begin{proof}
    This follows directly from computation and a straightforward application of \ref{Hilb:Tensor}. 
\end{proof}

Comparing discreteness in the von Neumann factors sense to discreteness in the C*-algebra sense, thus formally amounts to the following theorem:
\begin{thm}[{Theorem~\ref{thmalpha:vNvsC*Disc}}]\label{thm:vNvsC*Disc}
    Let $F:\cC\to \fgpBim(A)$ and $F'':\cC\to \spbfBim(N)$ be compatible actions, and $\bbB$ be a connected C*/W*-algebra object in $\Vec(\cC).$ 
    Then the map 
    \begin{align*}
        \Ad(U):\End_{-N}(L^2|\bbB|_{\bbF''})&\to \End_{-N}(L^2(B,\varphi))\\
        T&\mapsto U^*\circ T\circ U
    \end{align*} 
    implements a normal, unital, expectation preserving $*$-isomorphism 
    $$\left( (N,\tr)\overset{E^M_N}{\subset} M\right)\cong \left( (N,\tr)\overset{{E^B_A}''}{\subset} B''\right).$$
\end{thm}
\begin{proof}
For any $b\in B^\diamondsuit$, any $c\in \cC$ and all $\xi_{(c)}\otimes f_{(c)}$ we have that 
$$
U^*\circ(b\rhd)\circ U \left( \left(\xi_{(c)}\otimes f_{(c)}\right)\Omega^\varphi\right) = b\blacktriangleright \left((\xi_{(c)}\otimes f_{(c)})\Omega^\varphi\right), 
$$
since $U$ commutes with the left $B^\diamondsuit$-actions. 
Therefore for every $b\in B^\diamondsuit$ we have 
$$
\Ad(U)(b\rhd) = b\blacktriangleright,
$$
and therefore 
$$
\Ad(U)[B^\diamondsuit\rhd]= [B^\diamondsuit\blacktriangleright].
$$
In particular, $U$ is $A$-$A$ bimodular, and so $\Ad(U)[A\rhd]=[A\blacktriangleright].$ 
Since $\Ad(U)$ is WOT-continuous, it then follows that 
$$N\subseteq P:=\overline{\Ad(U^*)[B^\diamondsuit\blacktriangleright]}^{\mathsf{WOT}} = \Ad(U^*)[\overline{B^\diamondsuit\blacktriangleright}^{\mathsf{WOT}}]\subseteq \overline{[B^\circ\rhd]}^{\mathsf{WOT}}=M.$$

We now prove that the von Neumann algebra $P$  is a factor. If $x\in P'\cap P,$ in particular $x\in N'\cap M,$ and so multiplication by $x$ determines an element $L_x\in \Hom_{N-N}(N\to M\Omega) = \bbB(1_\cC)\cong \bbC.$ 
Thus, $x$ is a scalar, and so $P$ is a factor. 

We now show that $P=M$. Applying the Galois correspondence \cite[Corollary 7.14]{MR3948170} to the intermediate subfactor inclusion $N \subseteq P \subseteq M$, there exists a connected W*-algebra object $\bbP$ such that $\bbOne\subseteq \bbP \subseteq \bbB$, whose fibers are explicitly 
$$
\bbP(Z) = \Hom_{N-N}(\bbF''(Z)^\circ \to M\Omega) \qquad \forall Z\in \cC.
$$
However, by construction 
$$ \bbF(Z) \subset \bbF''(Z)^\circ\ \text{ densely }\qquad \forall Z\in\cC.$$
Thus, $\bbP= \bbB$ and therefore $P=M,$ as claimed. 

It follows that the *-isomorphism $\Ad(U)$ is a normal map preserving the inclusions: 
$$
\Ad(U)[N\subset M] = (A''\subset B'').
$$ 

Finally, $\Ad(U)$ is also expectation preserving since at the algebraic level $B^\diamondsuit$ we have that 
$$
 E^M_N(b\rhd) = E^B_A(\Ad(U)(b\rhd)) = E^B_A(b\blacktriangleright)= {E^B_A}''(b\blacktriangleright),
$$
given that $U$ preserves the grading. 
Since all maps above are normal, we conclude that $\Ad(U)$ is expectation preserving from $M$ onto $B''.$ 
\end{proof}

Theorem \ref{thm:vNvsC*Disc} in combination with Example \ref{ex:GJSHilbertification} yields the following corollary: 
\begin{cor}
For any unitary tensor category $\cC$ there is a one-to-one correspondence between connected C*-discrete extensions and connected von Neumann discrete extensions under the GJS actions. 
\end{cor}

\bigskip

Within the context of Theorem \ref{thm:vNvsC*Disc}, we shall now compare the lattice of intermediate subalgebras to a given irreducible vN-discrete extremal inclusion and its C*-counterpart: 
\begin{prop}\label{Prop:Galois}
    Let $N\subset M$ and $A\subset B$ satisfying conclusions of Theorem \ref{thm:vNvsC*Disc}. 
    There is a lattice injection: 
\begin{align*}
        \Psi:\left\{ P\in\vNA\ \middle|\ N\subseteq P\subseteq M  \right\}
        &\hookrightarrow \left\{ D\in\rCorr\ \middle|\ A\subseteq D\subseteq B  \right\}\\
        P&\mapsto A\rtimes_{F,r}\bbP. 
\end{align*}
    And a surjection
\begin{align*}
        \Phi: \left\{ D\in\rCorr\ \middle|\ A\subseteq D\subseteq B  \right\} &\twoheadrightarrow\left\{ P\in\vNA\ \middle|\ N\subseteq P\subseteq M  \right\}\\
        D&\mapsto D''. 
    \end{align*}
    Here, $\bbP$ is the connected W*/C* algebra object underlying $N\subseteq P$, and the commutants are taken within $\End(L^2(B,\varphi)).$
    Furthermore, these maps satisfy $$\Phi\circ\Psi = \id.$$
\end{prop}
\begin{proof}
    We shall systematically omit the maps $\Omega^a$ and $\eta^a$ from the definition of compatibility to ease notation and assume they are canonical inclusions. 
    Furthermore, the difference between taking double commutators in $L^2|\bbB|_{\bbF''}$ or $L^2(B, \varphi)$ is immaterial by Lemma \ref{lem:UnitaryNN}.
    
    Notice that the underlying algebra object $\bbP$ is well-defined by \cite[\S 7.3]{MR3948170}, which is moreover connected with $\bbOne\subseteq \bbP \subseteq \bbB$.
    And so, the definition of $\Psi$ is meaningful, giving a clearly injective map.

    By compatibility of actions, for each $Z\in\cC$ the inclusion $F(Z)\subseteq F''(Z)^\circ$ is dense. 
    Now, we claim 
    $$
    P= \overline{\oplus_Z F''(Z)^\circ\otimes \bbP(Z)}^{\sf{WOT}}\supset \overline{\oplus_Z F(Z)\otimes \bbP(Z)}^{\|\cdot\|_{\sf{op}}^{L^2(B,\varphi)}} =:D
    $$
    is {\sf WOT}-dense. 
    Otherwise, underlying $D''$ there is a connected W*-algebra object $\bbOne\subseteq\bbD\subseteq\bbP$ by the aforementioned Galois correspondence.  
    But if $\bbD\neq \bbP$ it would contradict that $F(Z)\subset F''(Z)^\circ$ is dense for all $Z\in \cC.$ 
    Thus, the W*-algebra object $\bbP$ is the same as the C*-algebra object $\bbD.$
    It follows that $\Phi\circ\Psi (P)=P$ and also that $\Phi$ is necessarily surjective.
\end{proof}

\begin{remark}
    Weather $\Psi\circ\Phi=\id$ is still unknown to us in this generality, as there might be intermediate C*-algebras $A\subseteq D\subseteq B$ which are not C*-discrete. 
    However, these has promising progress in this direction, where reasonable restraints have been found on the underlying actions making the Galois correspondence  tight; cf \cite[\S4]{2024arXiv240113989M}. 
\end{remark}

\bibliographystyle{amsalpha}

\bibliography{bibliography}

\providecommand{\bysame}{\leavevmode\hbox to3em{\hrulefill}\thinspace}
\providecommand{\MR}{\relax\ifhmode\unskip\space\fi MR }
\providecommand{\MRhref}[2]{%
  \href{http://www.ams.org/mathscinet-getitem?mr=#1}{#2}
}
\providecommand{\href}[2]{#2}
\begin{thebibliography}{CHPJP22}

\bibitem[AF15]{MR3431668}
David Ayala and John Francis, \emph{Factorization homology of topological
  manifolds}, J. Topol. \textbf{8} (2015), no.~4, 1045--1084. \MR{3431668}

\bibitem[AV20]{MR4162123}
Jamie Antoun and Christian Voigt, \emph{On bicolimits of {$C^*$}-categories},
  Theory Appl. Categ. \textbf{35} (2020), Paper No. 46, 1683--1725.
  \MR{4162123}

\bibitem[BHP12]{BHP12}
Arnaud Brothier, Michael Hartglass, and David Penneys, \emph{Rigid {$C\sp
  *$}-tensor categories of bimodules over interpolated free group factors}, J.
  Math. Phys. \textbf{53} (2012), no.~12, 123525, 43. \MR{3405915}

\bibitem[Bis97]{Bisch97}
D.~Bisch, \emph{Bimodules, higher relative commutants and the fusion algebra
  associated to a subfactor}, Fields Institute Communications \textbf{13}
  (1997), 13--63.

\bibitem[BO08]{MR2391387}
Nathanial~P. Brown and Narutaka Ozawa, \emph{{$C^*$}-algebras and
  finite-dimensional approximations}, Graduate Studies in Mathematics, vol.~88,
  American Mathematical Society, Providence, RI, 2008. \MR{2391387}

\bibitem[CHPJP22]{MR4419534}
Quan Chen, Roberto Hern\'{a}ndez~Palomares, Corey Jones, and David Penneys,
  \emph{Q-system completion for {$\rm C^*$} 2-categories}, J. Funct. Anal.
  \textbf{283} (2022), no.~3, Paper No. 109524, 59. \MR{4419534}

\bibitem[DCY13]{MR3121622}
Kenny De~Commer and Makoto Yamashita, \emph{Tannaka-{K}re\u{\i}n duality for
  compact quantum homogeneous spaces. {I}. {G}eneral theory}, Theory Appl.
  Categ. \textbf{28} (2013), No. 31, 1099--1138. \MR{3121622}

\bibitem[DCY15]{MR3420332}
\bysame, \emph{Tannaka-{K}re\u in duality for compact quantum homogeneous
  spaces {II}. {C}lassification of quantum homogeneous spaces for quantum {$\rm
  SU(2)$}}, J. Reine Angew. Math. \textbf{708} (2015), 143--171. \MR{3420332}

\bibitem[{Hat}23]{2023arXiv230407155H}
Lucas {Hataishi}, \emph{{C*-Algebraic Factorization Homology and Realization of
  Cyclic Representations}}, arXiv e-prints (2023), arXiv:2304.07155.

\bibitem[HHP20]{MR4139893}
Michael Hartglass and Roberto Hern\'{a}ndez~Palomares, \emph{Realizations of
  rigid {$\rm C^*$}-tensor categories as bimodules over {GJS} {$\rm
  C^*$}-algebras}, J. Math. Phys. \textbf{61} (2020), no.~8, 081703, 32.
  \MR{4139893}

\bibitem[HP25]{2025arXiv250321515H}
Roberto Hern\textbackslash'andez~Palomares, \emph{{Discrete inclusions from
  Cuntz-Pimsner algebras}}, arXiv e-prints (2025), arXiv:2503.21515.

\bibitem[HPN23]{2023arXiv230505072H}
Roberto Hern{\'a}ndez~Palomares and Brent {Nelson}, \emph{{Discrete Inclusions
  of C*-algebras}}, arXiv e-prints (2023), arXiv:2305.05072.

\bibitem[HPN24]{2024arXiv240918161H}
\bysame, \emph{{Remarks on C*-discrete inclusions}}, arXiv e-prints (2024),
  arXiv:2409.18161.

\bibitem[HY22]{2022arXiv220506663H}
Lucas {Hataishi} and Makoto {Yamashita}, \emph{{Injectivity for algebras and
  categories with quantum symmetry}}, arXiv e-prints (2022), arXiv:2205.06663.

\bibitem[ILP98]{MR1622812}
Masaki Izumi, Roberto Longo, and Sorin Popa, \emph{A {G}alois correspondence
  for compact groups of automorphisms of von {N}eumann algebras with a
  generalization to {K}ac algebras}, J. Funct. Anal. \textbf{155} (1998),
  no.~1, 25--63. \MR{1622812}

\bibitem[Jon83]{MR696688}
V.~F.~R. Jones, \emph{Index for subfactors}, Invent. Math. \textbf{72} (1983),
  no.~1, 1--25. \MR{696688}

\bibitem[Jon22]{MR4374438}
\bysame, \emph{Planar algebras, {I}}, New Zealand J. Math. \textbf{52} (2021
  [2021--2022]), 1--107. \MR{4374438}

\bibitem[JP17]{JP17}
Corey Jones and David Penneys, \emph{Operator algebras in rigid {$\rm
  C^*$}-tensor categories}, Comm. Math. Phys. \textbf{355} (2017), no.~3,
  1121--1188. \MR{3687214}

\bibitem[JP19]{MR3948170}
\bysame, \emph{Realizations of algebra objects and discrete subfactors}, Adv.
  Math. \textbf{350} (2019), 588--661. \MR{3948170}

\bibitem[JS97]{MR1473221}
V.~Jones and V.~S. Sunder, \emph{Introduction to subfactors}, London
  Mathematical Society Lecture Note Series, vol. 234, Cambridge University
  Press, Cambridge, 1997. \MR{1473221}

\bibitem[Kos86]{MR829381}
Hideki Kosaki, \emph{Extension of {J}ones' theory on index to arbitrary
  factors}, J. Funct. Anal. \textbf{66} (1986), no.~1, 123--140. \MR{829381}

\bibitem[KPW04]{MR2085108}
Tsuyoshi Kajiwara, Claudia Pinzari, and Yasuo Watatani, \emph{Jones index
  theory for {H}ilbert {$C^*$}-bimodules and its equivalence with conjugation
  theory}, J. Funct. Anal. \textbf{215} (2004), no.~1, 1--49. \MR{2085108}

\bibitem[M{\"u}g03a]{MR1966524}
Michael M{\"u}ger, \emph{From subfactors to categories and topology. {I}.
  {F}robenius algebras in and {M}orita equivalence of tensor categories}, J.
  Pure Appl. Algebra \textbf{180} (2003), no.~1-2, 81--157,
  \mathscinet{MR1966524} \doi{10.1016/S0022-4049(02)00247-5}
  \arXiv{math.CT/0111204}.

\bibitem[M{\"u}g03b]{MR1966525}
\bysame, \emph{From subfactors to categories and topology. {II}. the quantum
  double of tensor categories and subfactors}, J. Pure Appl. Algebra
  \textbf{180} (2003), no.~1-2, 159--219.

\bibitem[{Muk}24]{2024arXiv240113989M}
Miho {Mukohara}, \emph{{Inclusions of simple C$^*$-algebras arising from
  compact group actions}}, arXiv e-prints (2024), arXiv:2401.13989.

\bibitem[Nes14]{MR3426224}
Sergey Neshveyev, \emph{Duality theory for nonergodic actions}, M\"unster J.
  Math. \textbf{7} (2014), no.~2, 413--437. \MR{3426224}

\bibitem[NT13]{MR3204665}
Sergey Neshveyev and Lars Tuset, \emph{Compact quantum groups and their
  representation categories}, Cours Sp\'ecialis\'es [Specialized Courses],
  vol.~20, Soci\'et\'e Math\'ematique de France, Paris, 2013,
  \mathscinet{MR3204665}. \MR{3204665}

\bibitem[NY17]{MR3837600}
Sergey Neshveyev and Makoto Yamashita, \emph{Towards a classification of
  compact quantum groups of {L}ie type}, Operator algebras and
  applications---the {A}bel {S}ymposium 2015, Abel Symp., vol.~12, Springer,
  [Cham], 2017, pp.~231--263. \MR{3837600}

\bibitem[Pop94]{MR1278111}
Sorin Popa, \emph{Classification of amenable subfactors of type {II}}, Acta
  Math. \textbf{172} (1994), no.~2, 163--255. \MR{1278111}

\bibitem[Pop95]{MR1334479}
\bysame, \emph{An axiomatization of the lattice of higher relative commutants
  of a subfactor}, Invent. Math. \textbf{120} (1995), no.~3, 427--445.
  \MR{1334479}

\bibitem[Shl99]{Shl99}
Dimitri Shlyakhtenko, \emph{{$A$}-valued semicircular systems}, J. Funct. Anal.
  \textbf{166} (1999), no.~1, 1--47. \MR{1704661}

\end{thebibliography}
\Addresses
\end{document}